\documentclass[letter, 11pt, reqno]{amsart}
\usepackage[includeheadfoot, margin=1in]{geometry}
\usepackage{amsfonts}
\usepackage{amsmath}
\usepackage{amsthm}
\usepackage{color}
\usepackage[mathscr]{euscript}
\usepackage{bm}
\usepackage{amscd}
\usepackage{comment}
\usepackage[numbers,square]{natbib}

\newtheorem{Thm}{Theorem}[section]

\newtheorem{Lem}[Thm]{Lemma}

\newtheorem{Observation}[Thm]{Observation}

\theoremstyle{definition}
\newtheorem{Def}[Thm]{Definition}
\newtheorem{Ex}[Thm]{Example}

\theoremstyle{remark}
\newtheorem{Remark}[Thm]{Remark}

\newtheorem*{Ack}{Acknowledgements}
\newtheorem*{Proofpart1}{Case 1: $\phi^Y(\mu) \neq 1$}
\newtheorem*{Proofpart2}{Case 2: $\phi^Y (\mu ) =1$ and $\phi^Y (h ) \neq 1$}
\newtheorem*{Proofpart3}{Case 3: $\phi^Y (\mu ) =1$ and $\phi^Y (h ) = 1$}

\begin{document}

\title{Turaev torsion invariants of 3-orbifolds}
\author{Biji Wong}
%\date{\today}
\address{Department of Mathematics, Brandeis University, MS 050, 415 South Street, Waltham, MA 02453}
\email{wongb@brandeis.edu}
\urladdr{https://sites.google.com/a/brandeis.edu/bijiwong}
\maketitle

\begin{abstract}
We construct a combinatorial invariant of 3-orbifolds with singular set a link that generalizes the Turaev torsion invariant of 3-manifolds. We give several gluing formulas from which we derive two consequences. The first is an understanding of how the components of the invariant change when we remove a curve from the singular set. The second is a formula relating the invariant of the 3-orbifold to the Turaev torsion invariant of the underlying 3-manifold in the case when the singular set is a nullhomologous knot. 
\end{abstract}

\section{Introduction}
In \cite{Tu76}, \cite{Tu86}, and \cite{Tu90}, Turaev introduced a combinatorial invariant of compact, homology oriented 3-manifolds $M$ with $b_1(M) \geq 1$ that takes the form of a function on the set of Euler structures. In \cite{MT}, Meng and Taubes observed that when the 3-manifolds are thought of with their smooth structures, a component of Turaev's torsion invariant, the Milnor torsion invariant, can be realized as a version of the Seiberg-Witten invariant, a function on the set of Spin$^{c}$ structures. Building on their ideas, Turaev showed in \cite{Tu98} that after an identification of Euler and Spin$^{c}$ structures, the Turaev torsion and Seiberg-Witten invariants are in fact equivalent (up to sign). Separately, in \cite{B} Baldridge extended the Seiberg-Witten invariant to compact, homology oriented smooth 3-orbifolds $Y$ with $b_1(|Y|) \geq 1$ and singular set a link. Here $|Y|$ is the underlying 3-manifold of $Y$. Later in \cite{C}, Chen showed that the orbifold Seiberg-Witten invariant of $Y$ can always be recovered from the Seiberg-Witten invariant of $|Y|$, after an identification of the orbifold Spin$^{c}$ structures on $Y$ with the Spin$^{c}$ structures on $|Y|$.

The goal of this paper is to construct a combinatorial invariant of compact, homology oriented 3-orbifolds with singular set a link that generalizes the Turaev torsion invariant of 3-manifolds and is more sensitive to orbifold structures than Baldridge's orbifold Seiberg-Witten invariant.

The organization of the paper is as follows: In section 2, we review the theory of orbifolds and the definition of the Turaev torsion invariant. In section 3, we extend the notion of Euler structures to 3-orbifolds with singular set a link. See Definition \ref{orbiEuler} and Theorem \ref{orbiEulergen}. In section 4, we define the orbifold Turaev torsion invariant and show that it is indeed an invariant, namely independent of the choices made. See Definition \ref{torsiondefinition} and Theorem \ref{Thm: orbifold invariant}, respectively. In section 5, we give several gluing formulas for the orbifold Turaev torsion invariant, generalizing gluing formulas for the regular Turaev torsion invariant. See Theorem \ref{orbifold gluing} and Theorem \ref{gen}. In section 6, we determine how the components of the orbifold Turaev torsion invariant change when we remove a curve from the singular set. See Theorem \ref{app1}. We also give a formula relating the orbifold Turaev torsion invariant to the Turaev torsion invariant of the underlying 3-manifold, in the case when the singular set is a nullhomologous knot. See Theorem \ref{thm: appplication}. The formula will suggest that the orbifold Turaev torsion invariant can be used to detect orbifold structures in contrast to the orbifold Seiberg-Witten invariant. 

\begin{Ack} 
The author would like to thank Weimin Chen for suggesting this problem and for helpful conversations. The author is also grateful to Danny Ruberman for his support and advice and for carefully reading an early draft of the paper. A part of this work was funded by a NSF IGERT fellowship under grant number DGE-1068620 and the NSF FRG grant DMS-1065784.
\end{Ack}

\section{Background on 3-orbifolds and Turaev torsion invariants of 3-manifolds}

\subsection{3-orbifolds} \label{3-orbifolds} We review elements of the theory of 3-orbifolds, for details see \cite{BMP, S83, Th}. A \textit{3-orbifold} $Y$ is a Hausdorff, second-countable space $|Y|$ that is locally modeled on quotients of $\mathbb{R}^3$ by finite subgroups $G$ of $O(3)$. Specifically, there is an atlas $\{U_i, \phi_i\}$, consisting of connected open sets $U_i$ in $|Y|$ and homeomorphisms $\phi_i: \mathbb{R}^3/G_i \rightarrow U_i$, where $G_i$ is a finite subgroup of $O(3)$ that acts continuously and effectively. On each overlap $U_i \subset U_j$ we require a compatibility condition: there is an injective homomorphism $f_{ji}: G_i \rightarrow G_j$ and an embedding $\widetilde{\phi}_{ji}: \mathbb{R}^3 \rightarrow \mathbb{R}^3$, equivariant with respect to $f_{ji}$, such that the following diagram commutes:

\begin{center}
$\begin{CD}
\mathbb{R}^3 @>\widetilde{\phi}_{ji}>> \mathbb{R}^3\\
@VVqV @VVqV\\
\mathbb{R}^3 / G_i @>\phi_{ji}>> \mathbb{R}^3 / G_j\\
@VV\phi_iV @VV\phi_jV\\
U_i @>incl>> U_j
\end{CD}$
\end{center}

\noindent Here $\phi_{ji}$ is the induced map and $q$ is the quotient map. To each $y \in |Y|$, we can associate a group $G_y$, well-defined up to isomorphism: take any chart $U \cong \mathbb{R}^3 / G$ containing $y$, each lift $\widetilde{y}$ of $y$ gives an isotropy subgroup $G_{\widetilde{y}} \subset G$. All of these isotropy subgroups are conjugate, and so $G_y$ is defined to be this isomorphism class of groups. The singular set $\Sigma Y$ consists of points $y \in |Y|$ with $G_y \neq \boldsymbol{1}$. If $\Sigma Y = \emptyset$, then $Y$ is an honest 3-manifold. Note that \textit{3-orbifolds with boundary} are defined in a similar manner.

A \textit{map} between 3-orbifolds $Y_1$ and $Y_2$ is a map between the underlying spaces $|Y_1|$ and $|Y_2|$ that takes charts $U_1 \cong \mathbb{R}^3 / G_1$ into charts $U_2 \cong \mathbb{R}^3 / G_2$, and each restriction $U_1 \rightarrow U_2$ lifts to a map $ \mathbb{R}^3 \rightarrow  \mathbb{R}^3$ that is equivariant with respect to some homomorphism $G_1 \rightarrow G_2$.

An orbifold \textit{covering} of $Y$ is a 3-orbifold $Y'$ with a projection map $p: |Y'| \rightarrow |Y|$ between the underlying spaces, so that each chart neighborhood $U \cong \mathbb{R}^3 / G$ for $Y$ pulls back to a disjoint union of chart neighborhoods for $Y'$, each of the form $\mathbb{R}^3 / H$, where $H$ is a subgroup of $G$, and the chart homeomorphisms, together with $p$, fit inside a certain commutative diagram. In general, $p: |Y'| \rightarrow |Y|$ is not a covering map. As in the regular theory, the \textit{deck group} of an orbifold covering $p: |Y'| \rightarrow |Y|$ consists of orbifold maps $Y' \rightarrow Y'$ that respect $p$. Furthermore, given any 3-orbifold $Y$, we have the notion of an orbifold \textit{universal} cover $\widetilde{Y}$: an orbifold covering that orbifold-covers all other orbifold coverings. $\pi_1^{orb}(Y)$ is defined to be the deck group of $\widetilde{Y}$ and $H_1^{orb}(Y)$ is defined to be the abelianization of $\pi_1^{orb}(Y)$.

In this paper, the 3-orbifolds $Y$ are compact, connected, and oriented with singular set $\Sigma Y$ an oriented link $L_1 \cup \ldots \cup L_k$ and boundary $\partial Y = \emptyset$ or a union of tori. Centered around each $L_i$ is a neighborhood of the form $(S^1 \times D^2) / \mathbb{Z}_{\alpha_i}$, where $\mathbb{Z}_{\alpha_i}$ acts by rotations about the core. Let $E$ denote the complement of the interiors of these neighborhoods. Then $H_1^{orb}(Y) \cong H_1(E)/ \langle \mu_1^{\alpha_1}, \ldots, \mu_k^{\alpha_k} \rangle$, where $\mu_i$ is the meridian of $L_i$ oriented so that its linking number with $L_i$ is 1. We will be interested in the orbifold cover $\widehat{Y}$ of $Y$ with deck group $H_1^{orb}(Y)$. It can be constructed in the following way: start with the regular cover $\overline{E}$ of $E$ with deck group $H_1^{orb}(Y)$. Then canonically extend $\partial \overline{E}$ to cover $\bigcup\limits_{i=1}^k (S^1 \times D^2) / \mathbb{Z}_{\alpha_i}$. For details, see \cite[Chapter 2.2.2]{BMP}. We conclude Section \ref{3-orbifolds} with several examples:

\begin{Ex}
Let $\Sigma^n(K)$ be the $n$-fold cyclic branched cover of $S^3$ branched along $K$. Then there is a natural action of $\mathbb{Z}_n$ on $\Sigma^n(K)$, and the quotient space can be thought of as the 3-orbifold $(S^3, K, n)$, where the underlying space is $S^3$, the singular set is $K$, and for any point $y$ on $K$ $G_y \cong \mathbb{Z}_n$. Furthermore, $H_1^{orb}(S^3, K, n) \cong \mathbb{Z}_n$ and $\Sigma^n(K)$ is the  orbifold cover of $(S^3, K, n)$ with deck group $H_1^{orb}(S^3, K, n)$. In the notation above, $\widehat{(S^3, K, n)} = \Sigma^n(K)$. Note that $\widehat{(S^3, K, n)}$ is an honest 3-manifold.
\end{Ex}

\begin{Ex}
Let $Y$ denote an equivariant neighborhood $(S^1 \times \mathring{D^2}) / \mathbb{Z}_{\alpha}$. Then $Y$ is a 3-orbifold with singular set $S^1 \times \textbf{0}$, $H_1^{orb}(Y) \cong \mathbb{Z} \times \mathbb{Z}_{\alpha}$, and $\widehat{Y} = \mathbb{R} \times \mathring{D^2}$. Here $\widehat{Y}$ is also an honest 3-manifold.
\end{Ex}

\begin{Ex}
Let $Y$ be the 3-orbifold $(S^2, 2, 3, 5) \times S^1$, where $(S^2, 2, 3, 5)$ is the 2-orbifold with three singular points of multiplicities 2, 3, and 5. Then $H_1^{orb}(Y) \cong \mathbb{Z}$ and $\widehat{Y}$ is the 3-orbifold $(S^2, 2, 3, 5) \times \mathbb{R}$.
\end{Ex}

%\begin{Ex}
%Let $X$ be a 4-manifold with a $S^1$-action where the isotropy groups are finite. Then the quotient space $X/S^1$ can be thought of a 3-orbifold with singular set a link. 
%\end{Ex}

\subsection{Turaev torsion invariants of 3-manifolds}\label{torsion 3-manifolds}

We start by recalling the torsion of a chain complex. For details, see \cite{M66, Tu86, Tu01, Tu02}. Let $C = (\textbf{0} \rightarrow C_m \xrightarrow{\partial_{m-1}} C_{m-1} \xrightarrow{\partial_{m-2}} \cdots \xrightarrow{\partial_0} C_0 \rightarrow \textbf{0})$ be a chain complex of finite-dimensional vector spaces $C_i$ over a field $F$. Suppose $C$ and $H(C)$ are based: for each $i$ we have an ordered basis $c_i$ for $C_i$ and an ordered basis $\overline{h_i}$ for $H_i(C)$. Let $h_i$ be a representative for $\overline{h_i}$. Note $h_i$ is an ordered basis for $Ker( \partial_{i-1} : C_i \rightarrow C_{i-1})$. For each $i$, choose a sequence $b_i$ of vectors in $C_i$ with the property that $\partial_{i-1}(b_i)$ is an ordered basis in $Im(\partial_{i-1})$. Then for every $i$, the sequence $\partial_{i}(b_{i+1})h_ib_i$, gotten by concatenating, is an ordered basis for $C_i$. We can compare the given basis $c_i$ to this new basis. Let $[\partial_{i}(b_{i+1})h_ib_i / c_i]$ denote the determinant of the change of basis matrix from $c_i$ to $\partial_{i}(b_{i+1})h_ib_i$. 

\begin{Def}
The \textit{torsion} $\tau(C)$ of $C$ is defined to be 
\[(-1)^{|C|} \prod\limits_{i=0}^m  [\partial_{i}(b_{i+1})h_ib_i / c_i]^{(-1)^{i+1}} \in F - \{0\},
\]
where
\[
|C| = \sum\limits_{i=0}^m \Big( \sum\limits_{r=0}^i \textrm{dim } C_r\Big)\Big(\sum\limits_{r=0}^i \textrm{dim } H_r(C)\Big) \in \mathbb{Z}_2.
\]
\end{Def}

\begin{Remark}
$\tau(C)$ depends on the given bases for $C$ and $H(C)$, but not on the choices of $b_i$ and $h_i$.
\end{Remark}

\noindent If $C$ is acyclic, then each $H_i(C)=0$, and the definition of $\tau(C)$ simplifies to $\prod\limits_{i=0}^m  [\partial_{i}(b_{i+1})b_i / c_i]^{(-1)^{i+1}}$.

Let $M$ be a compact, connected, orientable 3-manifold with empty or toroidal boundary $\partial M$ and a fixed cell structure. Let $\widehat{M}$ denote the cover of $M$ with deck group $H_1(M)$. $\widehat{M}$ inherits a cell structure. Consider the cellular chain complex $C(\widehat{M})$ of $\widehat{M}$ with $\mathbb{Z}$ coefficients. The free action of $H_1(M)$ on the cells in $\widehat{M}$ gives $C(\widehat{M})$ the structure of a free $\mathbb{Z}[H_1(M)]$-chain complex. Let $I: \mathbb{Z}[H_1(M)] \hookrightarrow Q(\mathbb{Z}[H_1(M)])$ be the inclusion of $\mathbb{Z}[H_1(M)]$ into its quotient ring $Q(\mathbb{Z}[H_1(M)])$. Because $H_1(M)$ is a finitely generated abelian group, $Q(\mathbb{Z}[H_1(M)])$ splits in a canonical way (up to order of the factors) as a direct sum $\bigoplus\limits_{l=1}^r F_l$ of fields $F_l$, indexed by equivalence classes of characters of $Tor(H_1(M))$. For every $l$, we have the ring map $\phi_l: \mathbb{Z}[H_1(M)] \rightarrow F_l$ gotten by starting with $I$, applying the canonical splitting, and then taking the projection to the $l$th component. From each $\phi_l$, we get a free chain complex $C^{\phi_l}(\widehat{M})= C(\widehat{M}) \otimes_{\phi_l} F_l$ over the field $F_l$.

Now orient and order the cells in $M$. Pick a lift of the cells in $M$ to $\widehat{M}$. Each cell in the lift inherits an orientation, and the lift inherits an ordering. Then for every $l$, the chain complex $C^{\phi_l}(\widehat{M})$ is based. If $C^{\phi_l}(\widehat{M})$ is acyclic, then set $\tau^{\phi_l}(M)= \tau(C^{\phi_l}(\widehat{M}))$. Otherwise, set $\tau^{\phi_l}(M)=0$. Let $\tau(M)$ denote the resulting element $\tau^{\phi_1}(M) + \cdots + \tau^{\phi_r}(M)$ in $Q(\mathbb{Z}[H_1(M)])$. $\tau(M)$ depends on the orientation and order of the cells in $M$, and on the way we lift the cells to $\widehat{M}$. Changing any of these choices changes $\tau(M)$ by an element in $Q(\mathbb{Z}[H_1(M)])$ of the form $\pm h$ where $h \in H_1(M)$, and hence $\tau(M)$ is not well-defined.

There are a couple of ways to get around this. The classical approach is to think of $\tau(M)$ as an element of $Q(\mathbb{Z}[H_1(M)])/ \pm H_1(M)$; then $\tau(M)$ is well-defined. With this perspective, Milnor \cite{M62} showed that if $E$ denotes the exterior of a knot $K$ in $S^3$, then $\tau(E)$ is the Alexander polynomial of $K$, up to a factor. We take the second approach, due to Turaev \cite{Tu86, Tu90}. Here the ambiguity in $\tau(M)$ is removed by equipping $M$ with a homology orientation and an Euler structure. 

\begin{Def}
A \textit{homology orientation} $\omega$ on $M$ is an orientation of the $\mathbb{R}$-vector space $\bigoplus\limits_{n=0}^3 H_n(M, \mathbb{R})$.
\end{Def}

%\noindent If $\partial M = \emptyset$, then an orientation of $M$ induces a homology orientation on $M$. As a result, if $\partial M = \emptyset$, we will assume $M$ is oriented, and take the induced homology orientation. If $\partial M$ consists of tori, then there is no way to derive a homology orientation on $M$ from an orientation of $M$. In this case, we will need to pick a homology orientation on $M$.

\begin{Def}
An \textit{Euler structure} $\textbf{e}$ on $M$ is a lift of the cells in $M$ to $\widehat{M}$, considered up to the following equivalence: given lifts $\{\hat{e}_i\}_{i \in I}$ and $\{\hat{f}_i \}_{i \in I}$ in $\widehat{M}$ of the cells $\{e_i\}_{i \in I}$ in $M$, we say $ \{\hat{e}_i\}_{i \in I} \sim \{\hat{f} _i\}_{i \in I}$ if the product $\prod\limits_{i \in I} (\hat{f}_i / \hat{e}_i)^{(-1)^{\text{dim } e_i}} \in H_1(M)$ equals 1. Here $\hat{f}_i / \hat{e}_i$ denotes the unique element in $H_1(M)$ that takes $\hat{e}_i$ to $\hat{f}_i$.
\end{Def}

%\begin{Remark}
%Suppose $\partial M = \emptyset$. Fix an orientation and a Riemannian metric on $M$. Then the set $Eul(M)$ of Euler structures on $M$ is equivalent to the set of homotopy classes of nowhere vanishing vector fields on $M$, where the homotopy is only required to exist on the complement of a 3-ball in $M$. It is also equivalent to the set of $Spin^{c}$-structures on $M$. For details, see \cite{Tu90, Tu97}.
%\end{Remark}

\begin{Remark}
Let $Eul(M)$ denote the set of Euler structures on $M$. There is a free and transitive action of $H_1(M)$ on $Eul(M)$: if $h \in H_1(M)$ and $[\{\hat{e}_i\}_{i \in I}] \in Eul(M)$, then $h \cdot [\{\hat{e}_i\}_{i \in I}]$ is the Euler structure $[\{\hat{f}_i\}_{i \in I}]$ with the property that for all representatives $\{\hat{e}_i\}_{i \in I}$ of $[\{\hat{e}_i\}_{i \in I}]$ and $\{\hat{f}_i\}_{i \in I}$ of $[\{\hat{f}_i\}_{i \in I}]$, the product $\prod\limits_{i \in I} (\hat{f}_i / \hat{e}_i)^{(-1)^{\text{dim } e_i}} = h$. As a result, $Eul(M)$ can be thought of as a translate of $H_1(M)$.
\end{Remark}

Given a homology orientation $\omega$ and an Euler structure $\textbf{e}$ on $M$, we get a well-defined element $\tau(M, \textbf{e}, \omega) \in Q(\mathbb{Z}[H_1(M)])$ as follows. As above, orient and order the cells $\{e_i\}_{i \in I}$ in $M$. For every $n \in \{0, 1, 2, 3\}$, pick an ordered basis $\omega_n$ for the $\mathbb{R}$-vector space $H_n(M, \mathbb{R})$ so that the sequence $\{\omega_n \}_{n=0}^{3}$ realizes the homology orientation $\omega$. Our choices of orientation, order, and $\omega_n$'s base and homology base the cellular chain complex $C(M, \mathbb{R})$ of $M$ over $\mathbb{R}$, and allow us to compute the torsion $\tau(C(M, \mathbb{R}))$ of $C(M, \mathbb{R})$. Let $\tau_0$ denote the sign of $\tau(C(M, \mathbb{R}))$. Now choose a representative $\{\hat{e}_i\}_{i \in I}$ of $\textbf{e}$. Applying the above construction to this choice of orientation, order, and lift $\{\hat{e}_i\}_{i \in I}$ gives us the element $\tau(M)= \tau^{\phi_1}(Y) + \cdots + \tau^{\phi_r}(M) \in Q(\mathbb{Z}[H_1(M)])$. $\tau(M, \textbf{e}, \omega)$ is defined to be $\tau_0 \cdot \tau(M)$.

\begin{Thm}[\cite{Tu86, Tu90}]
$\tau(M, \textbf{e}, \omega)$ does not depend on the orientation and order of the cells in $M$, the sequence $\{\omega_n \}_{n=0}^{3}$ of bases realizing $\omega$, or on the representative $\{\hat{e}_i\}_{i \in I}$ of $\textbf{e}$.
\end{Thm}

\begin{Remark}
$\tau(M, \textbf{e}, \omega)$ does depend on $\textbf{e}$ and $\omega$.
\end{Remark}

\begin{Def}
Fixing $\omega$, we get a well-defined function $\tau: Eul(M) \rightarrow Q(\mathbb{Z}[H_1(M)])$ that sends an Euler structure $\textbf{e}$ to $\tau(M, \textbf{e}, \omega)$. The \textit{Turaev torsion invariant} of $M$ is $\tau$.
\end{Def}

\begin{Remark}
Let $M'$ denote $M$ with a different cell structure. Then there is a canonical identification $\theta : Eul(M') \rightarrow Eul(M)$, and the Turaev torsion invariant $\tau '$ of $M'$ equals $ \tau \circ \theta$. For details, see \cite[Chapters 1.2.1, 1.2.2]{Tu02}. As a result, we won't worry about the choice of cell structure.
\end{Remark}

We will need the following properties of $\tau$.

\begin{Thm}[{{\cite[Theorem 4.1]{Tu97}}}]\label{ThmTu97}
Assume $b_1 (M) \geq 2$. Then $\tau(M, \textbf{e}, \omega) \in \mathbb{Z}[H_1(M)]$ for every homology orientation $\omega$ and Euler structure $\textbf{e}$. Consequently, we will think of the invariant $\tau$ as a map $Eul(M) \rightarrow \mathbb{Z}[H_1(M)]$. Furthermore, if $F$ is a field and $\phi: \mathbb{Z}[H_1(M)] \rightarrow F$ is a ring homomorphism that is nontrivial on $H_1(M)$, then for every $\omega$ and $\textbf{e}$, the image of $\tau(M, \textbf{e}, \omega)$ under $\phi$ is the well-defined element $\tau_0 \cdot \tau^{\phi}(M)$, computed by picking any representative of $\textbf{e}$ and any sequence $\{\omega_n \}_{n=0}^{3}$ of bases realizing $\omega$. We will denote $\tau_0 \cdot \tau^{\phi}(M)$ by $\tau^{\phi}(M, \textbf{e}, \omega)$.
\end{Thm}

Now suppose $M$ is obtained by gluing a solid torus $S^1 \times D^2 \subset \mathbb{C} \times \mathbb{C}$ to a compact, connected, orientable 3-manifold $E$ with toroidal boundary. We will need two gluing formulas relating the Turaev torsion invariant $\tau_M$ of $M$ to the Turaev torsion invariant $\tau_E$ of $E$. To state them, we first need to explain how the homology orientations and Euler structures on $M$ are related to those on $E$:

Let $\omega$ be a homology orientation on $E$. Orient the core circle $S^1 \times \textbf{0}$ of the solid torus $S^1 \times D^2$. We get an induced homology orientation $\omega^M$ on $M$ as follows. First, fix an orientation of $\textbf{1} \times D^2$. This orients the $\mathbb{R}$-vector space $H_2(S^1 \times D^2, S^1 \times \partial D^2, \mathbb{R}) \cong \mathbb{R}$. By multiplying the orientation of $S^1 \times \textbf{0}$ with the orientation of $\textbf{1} \times D^2$, the solid torus $S^1 \times D^2$ inherits an orientation. This orients the $\mathbb{R}$-vector space $H_3(S^1 \times D^2, S^1 \times \partial D^2, \mathbb{R}) \cong \mathbb{R}$. Note that $H_n(S^1 \times D^2, S^1 \times \partial D^2, \mathbb{R}) = 0$ for $n \neq 2, 3$. By excision, $H_n(M, E, \mathbb{R}) \cong H_n(S^1 \times D^2, S^1 \times \partial D^2, \mathbb{R})$, and so we get an orientation $\omega^{(M,E)}$ of $\bigoplus\limits_{n=0}^3 H_n(M, E, \mathbb{R})$; we will think of $\omega^{(M,E)}$ as the induced homology orientation on $(M, E)$. Note that $\omega^{(M,E)}$ does not depend the choice of orientation for $\textbf{1} \times D^2$. Now consider the long exact sequence $\mathcal{H}$ of the pair $(M, E)$. There is a unique homology orientation $\widetilde{\omega^M}$ on $M$ so that the torsion $\tau(\mathcal{H})$ of $\mathcal{H}$ with respect to bases realizing the homology orientations $\omega$, $\omega^{(M,E)}$, $\widetilde{\omega^M}$ has positive sign. We define the homology orientation $\omega^M$ on $M$ induced by the homology orientation $\omega$ on $E$ to be $(-1)^{1+(b_1(E)+1)(b_1(M)+1)}\widetilde{\omega^M}$. Note that the sign $(-1)^{1+(b_1(E)+1)(b_1(M)+1)}$ is needed to ensure certain properties of $\omega^M$. For details, see \cite[Chapter 5.2]{Tu02}.

We now explain how the Euler structures are related. We assume that the solid torus $S^1 \times D^2$ is equipped with the following (open) cell decomposition: the boundary $S^1 \times \partial D^2$ is given the standard structure consisting of one 0-cell $(1, 1)$, two 1-cells $(S^1 - \textbf{1}) \times \textbf{1}$ and $\textbf{1} \times (\partial D^2 - \textbf{1})$, and one 2-cell $(S^1 - \textbf{1}) \times \partial (D^2 - \textbf{1})$, while the interior is given the cell decomposition consisting of one 0-cell $e^0 = (1, 0)$, two 1-cells $e^1_1 = \textbf{1} \times int ([0,1]) $ \& $e^1_2 = (S^1 - \textbf{1}) \times \textbf{0}$, two 2-cells $e^2_1 = (S^1 - \textbf{1}) \times int([0,1])$ \& $e^2_2 = \textbf{1} \times int(D^2)$, and one 3-cell $e^3 = (S^1 - \textbf{1}) \times int(D^2)$. This induces a cell decomposition of $\partial E$. Extend this to a cell decomposition of $E$, giving us a decomposition of $M$. Let $\textbf{e}$ be an Euler structure on $E$. Orient the core circle $S^1 \times \textbf{0}$ of the solid torus $S^1 \times D^2$. We get an induced Euler structure $\textbf{e}^M$ on $M$ as follows. First, from the orientation of $S^1 \times \textbf{0}$ we get a distinguished element $h \in H_1(M)$. Next, pick a lift $\{\hat{e}_j\} \subset \widehat{E}$ representing $\textbf{e}$. By covering space theory, we can always find a projection of $\widehat{E}$ to $\widehat{M}$ that is a lift of the inclusion $E \hookrightarrow M$ and is equivariant with respect to the induced homomorphism $H_1(E) \rightarrow H_1(M)$. Fix one of them. Then we can think of $\{\hat{e}_j\}$ as a lift of the cells in $E \subset M$ to $\widehat{M}$: over each cell in $E \subset M$ lies exactly one cell in $\{\hat{e}_j\} \subset \widehat{M}$. Now lift the cells $e^0, \ldots, e^3$ in the interior of $S^1 \times D^2 \subset M$ to cells $\hat{e}^0, \ldots, \hat{e}^3$ in $\widehat{M}$ so that $\partial (\hat{e}^1_2) = \pm (h-1)\hat{e}^0$, $\partial (\hat{e}^2_1) = \pm (h-1)\hat{e}^1_1 \pm \hat{e}^1_2$ modulo a 1-cell lying over $S^1 \times \partial D^2$, and $\partial (\hat{e}^3) = \pm (h-1)\hat{e}^2_2$ modulo a 2-cell lying over $S^1 \times \partial D^2$. We set $\textbf{e}^M$ to be the well-defined Euler structure represented by this family $\{\hat{e}_j\} \cup \{\hat{e}^0, \ldots, \hat{e}^3\}$ of lifts in $\widehat{M}$. 

\begin{Remark}
Our choice of cell structure on $S^1 \times D^2$ differs from the one in \cite{Tu02}: the core circle $S^1 \times \textbf{0}$ now forms a subcomplex. We will need this later.
\end{Remark}

The gluing formulas that we will need are as follows:

\begin{Thm}[{{\cite[Lemma 7.1.1 and Lemma 8.1.2]{Tu02}}}]\label{gluing lemmas}
Let $E$ be a compact, connected, orientable 3-manifold with $\partial E$ consisting of tori. Let $M$ be a 3-manifold obtained by gluing a solid torus $S^1 \times D^2$ to $E$ along a component of $\partial E$. Suppose $S^1 \times D^2$ is given the cell structure from above and that $E$ is given a compatible cell structure, inducing a cell structure on $M$. Fix an Euler structure $\textbf{e}$ and a homology orientation $\omega$ on $E$. This induces an Euler structure $\textbf{e}^M$ and a homology orientation $\omega^M$ on $M$. Orient $S^1 \times \textbf{0} \subset S^1 \times D^2$, and let $h \in H_1(M)$ denote the corresponding homology class. Let $F$ be a field, and let $\phi : \mathbb{Z}[H_1(E)] \rightarrow F$ be a ring homomorphism that extends to a ring homomorphism $\phi^M : \mathbb{Z}[H_1(M)] \rightarrow F$. We have a couple of cases:

\begin{enumerate}
\item Suppose $\phi^M (h) \neq 1$. Then $\tau ^{\phi^M} (M, \textbf{e}^M, \omega^M) = \tau^{\phi}(E, \textbf{e}, \omega) \cdot (\phi^M (h )- 1)^{-1}$.
\item Suppose $\phi^M (h) = 1$. Suppose further that $C^{\phi^M}(\widehat{M})$ is acyclic. Orient the meridian $\textbf{1} \times \partial D^2$ of $S^1 \times D^2$ so that its linking number with $S^1 \times \textbf{0}$ is 1. Let $e^2_2$ denote the 2-cell in $S^1 \times D^2$. Orient $e^2_2$ so that $\partial (e^2_2)= \textbf{1} \times \partial D^2$. Let $e^3$ denote the 3-cell $(S^1 - \textbf{1}) \times int(D^2)$ in $S^1 \times D^2$. Give $e^3$ the product orientation. Then we can lift $e^2_2$ to an oriented 2-cell $\hat{e}^2_2 \subset \widehat{M}$ and $e^3$ to an oriented 3-cell $\hat{e}^3 \subset \widehat{M}$ so that the homology classes $\big( \partial ( \hat{e}^2_2) \cap \widehat{E} \big) \otimes 1 \in H_{1}(C^{\phi}(\widehat{E}))$, $\big( \partial (\hat{e}^3) \cap \widehat{E} \big) \otimes 1 \in H_{2}(C^{\phi}(\widehat{E}))$ form a basis for $\bigoplus\limits_{i=0}^3 H_{i}(C^{\phi}(\widehat{E}))$. Furthermore, $\tau ^{\phi^M} (M, \textbf{e}^M, \omega^M) = \tau ^{\phi} \Big( E, \textbf{e}, \omega; \{ \big( \partial ( \hat{e}^2_2) \cap \widehat{E} \big) \otimes 1, \big( \partial (\hat{e}^3) \cap \widehat{E} \big) \otimes 1 \} \Big)$.
 \end{enumerate}
 \end{Thm}
 
 \begin{Remark}
In Case 2, $\tau ^{\phi} \Big( E, \textbf{e}, \omega; \{ \big( \partial ( \hat{e}^2_2) \cap \widehat{E} \big) \otimes 1, \big( \partial (\hat{e}^3) \cap \widehat{E} \big) \otimes 1 \} \Big)$ is $\tau_0$ times the torsion of $C^{\phi}(\widehat{E})$ with respect to the ordered basis $\{ \big( \partial ( \hat{e}^2_2) \cap \widehat{E} \big) \otimes 1, \big( \partial (\hat{e}^3) \cap \widehat{E} \big) \otimes 1 \}$ for $\bigoplus\limits_{i=0}^3 H_i (C^{\phi}(\widehat{E}))$, with $\tau_0$ defined as before. Also, we lose nothing by assuming $C^{\phi^M}(\widehat{M})$ is acyclic because if $C^{\phi^M}(\widehat{M})$ is not acyclic, then $\tau ^{\phi^M} (M, \textbf{e}^M, \omega^M) =0$.
\end{Remark}

\begin{Remark}
Orient $\textbf{1} \times \partial D^2 \subset S^1 \times D^2$. Let $\mu \in H_1(E)$ denote its induced homology class. Because $H_1(M) \cong H_1(E) / \langle \mu \rangle$, $\phi$ extends to $\phi^M$ when $\phi(\mu) =1$.
\end{Remark}
 
 \begin{Remark}
 Despite a different choice of cell structure on $S^1 \times D^2$, Theorem \ref{gluing lemmas} can be proved as in \cite{Tu02}.
 \end{Remark}
 
 \section{Orbifold Euler Structures}\label{orbifold Euler structures}
 
In this section we extend the notion of Euler structures to orbifolds. 

Let $Y$ be a compact, connected 3-orbifold with $\Sigma Y = L_1 \cup \ldots \cup L_k$. Centered around each $L_i$ is a neighborhood of the form $(S^1 \times D^2) / \mathbb{Z}_{\alpha_i}$, where $\mathbb{Z}_{\alpha_i}$ acts by rotations about the core. Equip each $(S^1 \times D^2) / \mathbb{Z}_{\alpha_i}$ with the cell decomposition of $S^1 \times D^2$ from Section \ref{torsion 3-manifolds}. In particular, each singular curve $L_i$ is given the cell decomposition consisting of a 0-cell $e^0_i = (1, 0)$ and a 1-cell $e^1_i = (S^1 - \textbf{1} ) \times \textbf{0}$. Then extend these cell decompositions to a cell decomposition of $|Y|$. Denote the set of cells away from $\Sigma Y$ by $\{e_j\}_{j \in J}$. 

The underlying space $|\widehat{Y}|$ of the orbifold cover $\widehat{Y}$ of $Y$ with deck group $H_1^{orb}(Y)$ inherits a cell decomposition. As in the regular case, $H_1^{orb}(Y)$ acts on the lifts of each cell in $|Y|$, but unlike the regular case, the action might not be free. For example, consider $(S^1 \times D^2) / \mathbb{Z}_\alpha$ with the above cell decomposition. $H_1^{orb}\big( (S^1 \times D^2) / \mathbb{Z}_\alpha \big) \cong \mathbb{Z} \times \mathbb{Z}_\alpha$. The $\mathbb{Z}_\alpha$ factor fixes the lifts of each cell in the singular curve $S^1 \times \textbf{0}$. More generally, for each $L_i$, the subgroup $\langle \mu_i \rangle$ of $H_1^{orb}(Y)$, generated by the meridian $\mu_i$ of $L_i$, fixes the lifts of each cell in $L_i$.

We define orbifold Euler structures on $Y$ in the following way. Instead of considering all possible lifts, as in the regular case, we restrict our attention to lifts that form a certain configuration over each singular curve. To formulate this precisely, first let $h_1, \ldots, h_k$ denote the homology classes in $H_1^{orb}(Y)$ induced by the oriented singular curves $L_1, \ldots, L_k$. If $\{\hat{e}_j\}_{j \in J} \cup \bigcup\limits_{i=1}^{k} \{\hat{e}^0_i, \hat{e}^1_i\}$ denotes a lift of the cells $\{e_j\}_{j \in J} \cup \bigcup\limits_{i=1}^{k} \{e^0_i, e^1_i\}$ in $|Y|$ to $|\widehat{Y}|$, then we require that $\partial (\hat{e}^1_i) = \pm (h_i-1)\hat{e}^0_i$ for every $i \in \{1, \ldots, k\}$. Given two such lifts $\hat{e} = \{\hat{e}_j\}_{j \in J} \cup \bigcup\limits_{i=1}^{k} \{\hat{e}^0_i, \hat{e}^1_i\}$ and $\hat{f} = \{\hat{f}_j\}_{j \in J} \cup \bigcup\limits_{i=1}^{k} \{\hat{f}^0_i, \hat{f}^1_i\}$, define 
\begin{equation}\label{equivrel}
\hat{f} / \hat{e} = \prod\limits_{j \in J} (\hat{f}_j / \hat{e}_j)^{(-1)^{\text{dim } e_j}},
\end{equation}
where $\hat{f}_j / \hat{e}_j$ is the unique element in $H_1^{orb}(Y)$ that takes $\hat{e}_j$ to  $\hat{f}_j$. Set $\hat{e} \sim \hat{f}$ when $\hat{f} / \hat{e} =1$. It is not hard to see that this gives an equivalence relation on the set of all such lifts. 

\begin{Remark}
We omit the product $\prod\limits_{i=1}^k( \hat{f}^0_i / \hat{e}^0_i) \cdot (\hat{f}^1_i / \hat{e}^1_i)^{-1}$ from the definition of $\hat{f} / \hat{e}$ because $\hat{f}^0_i / \hat{e}^0_i$ and $\hat{f}^1_i / \hat{e}^1_i$ may not be well-defined for some $i$. When $\hat{f}^0_i / \hat{e}^0_i$ and $\hat{f}^1_i / \hat{e}^1_i$ are well-defined, $(\hat{f}^0_i / \hat{e}^0_i) \cdot (\hat{f}^1_i / \hat{e}^1_i)^{-1} =1$ by definition of the configuration.
\end{Remark}

\begin{Def}\label{orbiEuler}
Let $Eul(Y)$ denote the set of equivalence classes. An \textit{orbifold Euler structure} $\textbf{e}$ on $Y$ is an element of $Eul(Y)$.
\end{Def}

\begin{Remark}
$Eul(Y)$ can be thought of as lifts to $|\widehat{Y}|$ of the cells away from $\Sigma Y$ modulo Relation \ref{equivrel} above. 
\end{Remark}

As in the regular case, we have the following:

\begin{Lem}
There is a free and transitive action of $H_1^{orb}(Y)$ on $Eul(Y)$: if $h \in H_1^{orb}(Y)$ and $\textbf{e} \in Eul(Y)$, then $h \cdot \textbf{e}$ is the orbifold Euler structure $\textbf{f} \in Eul(Y)$ with the property that $\hat{f} / \hat{e} = h$, for all representatives $\hat{e}$ of $\textbf{e}$ and $\hat{f}$ of $\textbf{f}$.
\end{Lem}

\begin{Remark}
This induces an action of $H_1^{orb}(Y)$ on classes of lifts to $|\widehat{Y}|$ away from $\Sigma Y$. 
\end{Remark}

Orbifold Euler structures generalize regular Euler structures:

\begin{Thm}\label{orbiEulergen}
Suppose $H_1^{orb}(Y) \cong H_1(|Y|)$. Then we have a canonical bijection $Eul(Y) \leftrightarrow Eul(|Y|)$.
\end{Thm}

\begin{proof}
$|\widehat{Y}|$ can be thought of as the regular cover $\widehat{|Y|}$ of $|Y|$ with deck group $H_1(|Y|)$, since $H_1^{orb}(Y) \cong H_1(|Y|)$. As a result, we can identify $Eul(Y)$ with the set $S'$ of lifts to $\widehat{|Y|}$ of the cells in $|Y|$ that form a certain configuration over each singular curve modulo Relation \ref{equivrel} above. Let $S$ denote the set of all lifts to $\widehat{|Y|}$ of the cells in $|Y|$. Let $I: S' \rightarrow S$ be the inclusion. We claim that $I$ induces a well-defined function $\overline{I}: Eul(Y) \rightarrow Eul(|Y|)$. Let $\{\hat{e}_j\}_{j \in J} \cup \bigcup\limits_{i=1}^{k} \{\hat{e}^0_i, \hat{e}^1_i\}$ and $\{\hat{f}_j\}_{j \in J} \cup \bigcup\limits_{i=1}^{k} \{\hat{f}^0_i, \hat{f}^1_i\}$ be lifts representing the same orbifold Euler structure on $Y$. Because $H_1^{orb}(Y) \cong H_1(|Y|)$, $\hat{f}^0_i / \hat{e}^0_i$ and $\hat{f}^1_i / \hat{e}^1_i$ are well-defined for every $i$. As noted above, the definition of the configuration guarantees that $(\hat{f}^0_i / \hat{e}^0_i) \cdot (\hat{f}^1_i / \hat{e}^1_i)^{-1} =1$ for every $i$. Hence $\prod\limits_{j \in J} (\hat{f}_j / \hat{e}_j)^{(-1)^{\text{dim } e_j}} \prod\limits_{i=1}^k( \hat{f}^0_i / \hat{e}^0_i) \cdot (\hat{f}^1_i / \hat{e}^1_i)^{-1} =1$, as needed. It is not hard to check that $\overline{I}$ is equivariant with respect to the free and transitive $H_1^{orb}(Y)$ and $H_1(|Y|)$ actions, hence $\overline{I}$ is a bijection.
\end{proof}

\section{Orbifold Turaev torsion invariants}\label{torsion 3-orbifolds}

In this section we extend the notion of Turaev torsion to orbifolds.

As in Section \ref{orbifold Euler structures}, $Y$ denotes a compact, connected 3-orbifold with $\Sigma Y = L_1 \cup \ldots \cup L_k$. Centered around each $L_i$ is a neighborhood of the form $(S^1 \times D^2) / \mathbb{Z}_{\alpha_i}$. Fix a cell decomposition on $|Y|$ that restricts to the preferred cell decomposition from Section \ref{orbifold Euler structures} on each neighborhood $(S^1 \times D^2) / \mathbb{Z}_{\alpha_i}$. This lifts to a cell decomposition of the underlying space $|\widehat{Y}|$ of the orbifold cover $\widehat{Y}$. 

Let $\textbf{e}$ be an orbifold Euler structure on $Y$, and let $\omega$ be a homology orientation on $|Y|$. Our definition of $\tau(Y, \textbf{e}, \omega)$ follows the regular construction with one difference: we have to be careful about how we order the cells in $|Y|$.

First, let $C(|\widehat{Y}|)$ denote the cellular chain complex of $|\widehat{Y}|$ with $\mathbb{Z}$ coefficients. The action of $H_1^{orb}(Y)$ on the lifts of each cell in $|Y|$ gives $C(|\widehat{Y}|)$ the structure of a $\mathbb{Z}[H_1^{orb}(Y)]$-chain complex. Note that the $\mathbb{Z}[H_1^{orb}(Y)]$-modules $C_0(|\widehat{Y}|)$ and $C_1(|\widehat{Y}|)$ may not be free because $H_1^{orb}(Y)$ may not act freely on the cells over the singular curves. 

Next, decompose $Q(\mathbb{Z}[H_1^{orb}(Y)])$ as a direct sum $\bigoplus\limits_{l=1}^r F_l$ of fields $F_l$, indexed by equivalence classes of characters of $Tor(H_1^{orb}(Y))$. For each $l$, we have the composition $\phi_l: \mathbb{Z}[H_1^{orb}(Y)] \rightarrow F_l$ gotten by starting with the inclusion $I: \mathbb{Z}[H_1^{orb}(Y)] \hookrightarrow Q(\mathbb{Z}[H_1^{orb}(Y)])$, applying the splitting, and then taking the projection to $F_l$. For each $l$, form the twisted chain complex $C^{\phi_l}(|\widehat{Y}|)= {\phi_l} C(|\widehat{Y}|) \otimes F_l$ over $F_l$. Note that the $F_l$-vector spaces $C^{\phi_l}_0(|\widehat{Y}|)$ and $C^{\phi_l}_1(|\widehat{Y}|)$ may have smaller than expected dimensions because $C_0(|\widehat{Y}|)$ and $C_1(|\widehat{Y}|)$ may not be free. Specifically, we have the following:

\begin{Observation}\label{Core Observation}
Fix $i \in \{1, \ldots, k\}$. Fix $\alpha \in \{0,1\}$. Let $\hat{e}^{\alpha}_i$ denote a lift of the $\alpha$-cell $e^{\alpha}_i$ in $L_i$ to $|\widehat{Y}|$. If $\phi_l(\mu_i ) \neq 1$, then $\hat{e}^{\alpha}_i \otimes 1 =0$ in $C^{\phi_l}_{\alpha}(|\widehat{Y}|)$.
\end{Observation}

\begin{proof}
$\hat{e}^{\alpha}_i \otimes 1 = \hat{e}^{\alpha}_i \otimes \phi_l (\mu_i - 1) \cdot (\phi_l (\mu_i -1))^{-1}= (\mu_i - 1) \cdot \hat{e}^{\alpha}_i \otimes (\phi_l (\mu_i -1))^{-1} = 0 \otimes (\phi_l (\mu_i -1))^{-1}  = 0.$
\end{proof}

Now order and orient the cells in $|Y|$. Then pick a lift in $|\widehat{Y}|$ that represents the orbifold Euler structure $\textbf{e}$. Each cell in the lift inherits an orientation, and cells of the same dimension inherit an ordering. Then for every $l$, the $F_l$-chain complex $C^{\phi_l}(|\widehat{Y}|)$ is based. For every $n \in \{0, 1, 2, 3\}$ pick an ordered basis $\omega_n$ for the $\mathbb{R}$-vector space $H_n(|Y|, \mathbb{R})$ so that the sequence $\{\omega_n \}_{n=0}^{3}$ realizes the homology orientation $\omega$. Our choices of orientation, order, and $\omega_n$'s base and homology base the cellular chain complex $C(|Y|, \mathbb{R})$ of $|Y|$ over $\mathbb{R}$, and allow us to compute the torsion $\tau(C(|Y|, \mathbb{R}))$ of $C(|Y|, \mathbb{R})$. Let $\tau_0$ denote the sign of $\tau(C(|Y|, \mathbb{R}))$. If $C^{\phi_l}(|\widehat{Y}|)$ is acyclic, set $\tau^{\phi_l}(Y, \textbf{e}, \omega)= \tau_0 \cdot \tau(C^{\phi_l}(|\widehat{Y}|))$. Otherwise, set $\tau^{\phi_l}(Y, \textbf{e}, \omega)=0$. 

\begin{Thm} \label{Thm: orbifold invariant}
$\tau^{\phi_l}(Y, \textbf{e}, \omega)$ does not depend on the orientation and order of the cells away from $\Sigma Y$, the lift in $|\widehat{Y}|$ representing $\textbf{e}$, or on the sequence $\{\omega_n \}_{n=0}^{3}$ of bases realizing $\omega$.
\end{Thm}

\begin{Remark}
The orientation and order of the singular curves in $\Sigma Y$ induce a natural orientation and order of the cells in $\Sigma Y$. Thus it suffices to focus on the orientation and order of the cells away from $\Sigma Y$.
\end{Remark}

The proof of Theorem \ref{Thm: orbifold invariant} makes use of the following:

\begin{Lem}[\cite{Tu01}]\label{Lem: chaingluing}
Let $C = (\textbf{0} \rightarrow C_m \rightarrow \ldots \rightarrow C_0 \rightarrow \textbf{0})$ be an acyclic chain complex of finite-dimensional vector spaces $C_i$ over a field $F$. If $C$ is based by $\{c_i\}$ and $\{d_i\}$, then $\tau(C, \{d_i\}) = \tau(C, \{c_i\}) \cdot \prod\limits_{i=0}^m [c_i / d_i]^{(-1)^{i+1}}$.
\end{Lem}

\begin{Lem}\label{Lem: orbifold invariant}
Let $\{\hat{e}_j\}_{j \in J} \cup \bigcup\limits_{i=1}^{k} \{\hat{e}^0_i, \hat{e}^1_i\}$ be any lift of the cells $\{e_j\}_{j \in J} \cup \bigcup\limits_{i=1}^{k} \{e^0_i, e^1_i\}$ in $|Y|$ to $|\widehat{Y}|$. Fix $\alpha \in \{0, 1, 2, 3\}$. Let $\{\hat{e}_s\}_{s \in S_{\alpha}}$ denote the set of $\alpha$-cells in $\{\hat{e}_j\}_{j \in J}$. Let $I' = \{1, \ldots, k \mid \phi_l( \mu_i ) =1\}$. If $\alpha \in \{0, 1\}$, then $\{\hat{e}_s \otimes 1\}_{s \in S_{\alpha}} \cup \{\hat{e}^{\alpha}_i \otimes 1 \}_{i \in I'}$ is a basis for $C^{\phi_l}_{\alpha}(|\widehat{Y}|)$. If $\alpha \in \{2, 3\}$, then $\{\hat{e}_s \otimes 1\}_{s \in S_{\alpha}}$ is a basis for $C^{\phi_l}_{\alpha}(|\widehat{Y}|)$.
\end{Lem}

\begin{proof}
The argument for $\alpha \in \{2, 3\}$ is similar to the one in the regular case because $\Sigma Y$ doesn't contain any 2-cells or 3-cells. Let $\alpha \in \{0, 1\}$. Given Observation \ref{Core Observation}, it is clear $\{\hat{e}_s \otimes 1, \hat{e}^{\alpha}_i \otimes 1 \}_{s \in S_{\alpha}, i \in I'}$  generate $C^{\phi_l}_{\alpha}(|\widehat{Y}|)$, so we will focus on linear independence. Suppose $\sum\limits_{s \in S_{\alpha}} q_s \cdot (\hat{e}_s \otimes 1) + \sum\limits_{i \in I'} q_i \cdot (\hat{e}^{\alpha}_i \otimes 1) =0$ for some $q_s, q_i \in F_l$. Fix $s_0 \in S_{\alpha}, i_0 \in I'$. We need to show $q_{s_0}, q_{i_0} = 0$. We show it for $q_{s_0}$, and the other case is similar. Let $C_{\alpha}'(|\widehat{Y}|)$ be the $\mathbb{Z}[H_1^{orb}(Y)]$-submodule of $C_{\alpha}(|\widehat{Y}|)$ generated by $\{\hat{e}_s\}_{s \in S_{\alpha}} \cup \{\hat{e}^{\alpha}_i\}_{i \in I'}$. Consider the well-defined function $\psi_{s_0} : C_{\alpha}'(|\widehat{Y}|) \times F_l \rightarrow F_l$ given by $(\sum\limits_{s \in S_{\alpha}} r_s \cdot \hat{e}_s + \sum\limits_{i \in I'} r_i \cdot \hat{e}^{\alpha}_i, f) \mapsto \phi_l(r_{s_0})f$. It is not hard to see that $\psi_{s_0}$ is $\mathbb{Z}[H_1^{orb}(Y)]$-balanced, and so $\psi_{s_0}$ extends to a $F_l$-linear map $\Psi_{s_0} : C_{\alpha}'(|\widehat{Y}|) \otimes F_l \rightarrow F_l$. Note that $\Psi_{s_0} (\hat{e}_{s_0} \otimes 1) = 1$, $\Psi_{s_0} (\hat{e}_{s} \otimes 1) = 0$ for $s \neq s_0$, and $\Psi_{s_0} (\hat{e}^{\alpha}_i \otimes 1) = 0$ for $i \in I'$. If we apply $\Psi_{s_0}$ to both sides of $\sum\limits_{s \in S_{\alpha}} q_s \cdot (\hat{e}_s \otimes 1) + \sum\limits_{i \in I'} q_i \cdot (\hat{e}^{\alpha}_i \otimes 1) =0$, then we get that $q_{s_0} =0$, as needed.

 \end{proof}

\begin{proof}[Proof of Theorem \ref{Thm: orbifold invariant}]
Assume $C^{\phi_l}(|\widehat{Y}|)$ is acyclic; otherwise there is nothing to prove. Let 
\[ 
\hat{e} = \{\hat{e}_j\}_{j \in J} \cup \bigcup\limits_{i=1}^{k} \{\hat{e}^0_i, \hat{e}^1_i\}
\]
and
\[ 
\hat{f} = \{\hat{f}_j\}_{j \in J} \cup \bigcup\limits_{i=1}^{k} \{\hat{f}^0_i, \hat{f}^1_i\} 
\]
be representatives of $\textbf{e}$. Let $\{\hat{e}_s\}_{s \in S_{\alpha}}$ denote the set of $\alpha$-cells in $\{\hat{e}_j\}_{j \in J}$, and let $\{\hat{f}_s\}_{s \in S_{\alpha}}$ denote the set of $\alpha$-cells in $\{\hat{f}_j\}_{j \in J}$. From Lemma \ref{Lem: orbifold invariant}, we have that $ \hat{e}_{\alpha} \otimes 1  = \{\hat{e}_s \otimes 1\}_{s \in S_{\alpha}} \cup \{\hat{e}^{\alpha}_i \otimes 1 \}_{i \in I'}$ and $\hat{f}_{\alpha} \otimes 1  = \{\hat{f}_s \otimes 1\}_{s \in S_{\alpha}} \cup \{\hat{f}^{\alpha}_i \otimes 1 \}_{i \in I'}$ are bases for $C^{\phi_l}_{\alpha}(|\widehat{Y}|)$ when $\alpha \in \{0, 1\}$, and that $\hat{e}_{\alpha} \otimes 1  = \{\hat{e}_s \otimes 1\}_{s \in S_{\alpha}}$ and $\hat{f}_{\alpha} \otimes 1  = \{\hat{f}_s \otimes 1\}_{s \in S_{\alpha}}$ are bases for $C^{\phi_l}_{\alpha}(|\widehat{Y}|)$ when $\alpha \in \{2, 3\}$. When $\alpha \in \{2, 3\}$, the matrix that takes $\hat{e}_{\alpha} \otimes 1$ to $\hat{f}_{\alpha} \otimes 1$ is diagonal with determinant 
\[
[\hat{f}_{\alpha} \otimes 1 / \hat{e}_{\alpha} \otimes 1] = \prod\limits_{s \in S_{\alpha}} \phi_l (\hat{f}_s / \hat{e}_s).
\]
Recall that $\hat{f}_s / \hat{e}_s$ is the unique element in $H_1^{orb}(Y)$ that takes $\hat{e}_s$ to $\hat{f}_s$. When $\alpha \in \{0, 1\}$, the matrix that takes $\hat{e}_{\alpha} \otimes 1$ to $\hat{f}_{\alpha} \otimes 1$ is diagonal with determinant 
\[
[\hat{f}_{\alpha} \otimes 1 / \hat{e}_{\alpha} \otimes 1] = \prod\limits_{s \in S_{\alpha}} \phi_l (\hat{f}_s / \hat{e}_s) \cdot \prod\limits_{i \in I'} \phi_l (\hat{f}^{\alpha}_i / \hat{e}^{\alpha}_i).
\]
Note that $\phi_l (\hat{f}^{\alpha}_i / \hat{e}^{\alpha}_i)$ is well-defined, even though $\hat{f}^{\alpha}_i / \hat{e}^{\alpha}_i$ is only defined up to powers of $\mu_i$. Then 
\begin{align*}
 \prod\limits_{\alpha = 0}^3 [\hat{f}_{\alpha} \otimes 1 / \hat{e}_{\alpha} \otimes 1]^{(-1)^{\alpha + 1}} &= \prod\limits_{\alpha =2}^3 \Big( \prod\limits_{s \in S_{\alpha}} \phi_l (\hat{f}_s / \hat{e}_s) \Big)^{(-1)^{\alpha +1}} \cdot \prod\limits_{\alpha =0}^1 \Big( \prod\limits_{s \in S_{\alpha}} \phi_l (\hat{f}_s / \hat{e}_s) \cdot \prod\limits_{i \in I'} \phi_l (\hat{f}^{\alpha}_i / \hat{e}^{\alpha}_i) \Big)^{(-1)^{\alpha +1}} \\
 &= \phi_l \big( \prod\limits_{\alpha =0}^3 \prod\limits_{s \in S_{\alpha}} (\hat{f}_s / \hat{e}_s) ^{(-1)^{\alpha +1}} \big) \cdot \prod\limits_{i \in I'} \prod\limits_{\alpha =0}^1 \big( \phi_l ( \hat{f}^{\alpha}_i / \hat{e}^{\alpha}_i ) \big) ^{(-1)^{\alpha +1}}
\end{align*}
Because $\hat{e}$ and $\hat{f}$ are in the same equivalence class, $\prod\limits_{\alpha =0}^3 \prod\limits_{s \in S_{\alpha}} (\hat{f}_s / \hat{e}_s) ^{(-1)^{\alpha +1}} =1$. Furthermore, because of our choice of configuration over each singular curve, $\phi_l (\hat{f}^0_i / \hat{e}^0_i) \cdot \big( \phi_l (\hat{f}^1_i / \hat{e}^1_i) \big)^{-1} =1$ for every $i \in I'$. Hence $\prod\limits_{\alpha = 0}^3 [\hat{f}_{\alpha} \otimes 1 / \hat{e}_{\alpha} \otimes 1]^{(-1)^{\alpha + 1}} =1$. By Lemma \ref{Lem: chaingluing}, $\tau (C^{\phi_l}(|\widehat{Y}|), \hat{e} \otimes 1) = \tau (C^{\phi_l}(|\widehat{Y}|), \hat{f} \otimes 1)$. Since the definition of $\tau_0$ does not involve taking lifts to $|\widehat{Y}|$, we have that $\tau^{\phi_l}(Y, \textbf{e}, \omega)$ does not depend on the lift in $|\widehat{Y}|$ representing $\textbf{e}$. The argument that $\tau^{\phi_l}(Y, \textbf{e}, \omega)$ does not depend on the way we orient the cells away from $\Sigma Y$ is similar to the one in the regular case: use Lemma \ref{Lem: chaingluing} and the fact that multiplying a column of a matrix by -1 changes the determinant by -1. Similarly, we can use the argument in the regular case to show that $\tau^{\phi_l}(Y, \textbf{e}, \omega)$ does not depend on the way we order the cells away from $\Sigma Y$: use Lemma \ref{Lem: chaingluing} and the fact that swapping two columns of a matrix changes the determinant by -1. Finally, the fact that $\tau^{\phi_l}(Y, \textbf{e}, \omega)$ does not depend on the sequence $\{\omega_n \}_{n=0}^{3}$ of bases realizing $\omega$ follows from the regular case.
\end{proof}

\begin{Def}\label{torsiondefinition}
Let $\tau(Y, \textbf{e}, \omega)$ denote $\tau^{\phi_1}(Y, \textbf{e}, \omega) + \ldots + \tau^{\phi_r}(Y, \textbf{e}, \omega) \in Q(\mathbb{Z}[H_1^{orb}(Y)])$. Fixing $\omega$, we get a well-defined function $\tau: Eul(Y) \rightarrow Q(\mathbb{Z}[H_1^{orb}(Y)])$ that sends an orbifold Euler structure $\textbf{e}$ to $\tau(Y, \textbf{e}, \omega)$. We call $\tau$ the \textit{orbifold Turaev torsion invariant} of $Y$.
\end{Def}

\begin{Remark}
For a different cell decomposition $Y'$ on $|Y|$ satisfying the same property, there is a canonical identification $\theta : Eul(Y') \rightarrow Eul(Y)$, and the orbifold Turaev torsion invariant $\tau'$ of $Y'$ equals $\tau \circ \theta$. The proof is similar to the argument in the regular case. 
\end{Remark}

%\subsection{A concrete example: \boldmath{$S^1 \times D^2 / \mathbb{Z}_3$}}

\section{Orbifold Gluing formulas}\label{Orbifold Gluing}

In this section, we give several gluing formulas for orbifold Turaev torsion.

\begin{Thm}\label{orbifold gluing}
Let $E$ be a compact, connected, oriented 3-orbifold with $\Sigma E$ an oriented link and $\partial E$ a union of tori. Glue an equivariant solid torus $(S^1 \times D^2) / \mathbb{Z}_{\alpha}$ to $E$ along a component of $\partial E$. We get a 3-orbifold $Y$ with $\Sigma Y = \Sigma E \cup S^1 \times \textbf{0}$. Fix an orbifold Euler structure $\textbf{e}$ on $E$ and a homology orientation $\omega$ on $|E|$. As in the regular case, this induces an orbifold Euler structure $\textbf{e}^Y$ on $Y$ and a homology orientation $\omega^{|Y|}$ on $|Y|$. Orient $S^1 \times \textbf{0}$, and let $h \in H_1^{orb}(Y)$ denote the induced homology class. Then orient the corresponding meridian $\textbf{1} \times (\partial D^2 /  \mathbb{Z}_{\alpha})$ so that its linking number with $S^1 \times \textbf{0}$ is 1. Let $\mu \in H_1^{orb}(Y)$ denote its induced homology class. Let $F$ be a field, and let $\phi: \mathbb{Z}[H_1^{orb}(E)] \rightarrow F$ be a ring homomorphism that extends to a ring homomorphism $\phi^Y : \mathbb{Z}[H_1^{orb}(Y)] \rightarrow F$. We have several cases:

\begin{enumerate}
\item Suppose $\phi^Y (\mu ) \neq 1$. Then $\tau^{\phi^Y} (Y, \textbf{e}^{Y}, \omega^{|Y|}) = \tau^{\phi} (E, \textbf{e}, \omega)$.
\item Suppose $\phi^Y (\mu ) =1$ and $\phi^Y (h ) \neq 1$. Then $\tau^{\phi^Y} (Y, \textbf{e}^{Y}, \omega^{|Y|}) = \tau^{\phi} (E, \textbf{e}, \omega) \cdot (\phi^Y (h) -1)^{-1}$.
\item Suppose $\phi^Y (\mu ) =1$ and $\phi^Y (h ) = 1$. Suppose further that $C^{\phi^Y}(|\widehat{Y}|)$ is acyclic. Let $e^2_2$ denote the 2-cell in $(S^1 \times D^2) / \mathbb{Z}_{\alpha}$. Orient $e^2_2$ so that $\partial (e^2_2) = \textbf{1} \times (\partial D^2 /  \mathbb{Z}_{\alpha})$. Let $e^3$ denote the 3-cell $(S^1 - \textbf{1}) \times (int(D^2) / \mathbb{Z}_{\alpha})$ in $(S^1 \times D^2) / \mathbb{Z}_{\alpha}$. Give $e^3$ the product orientation. Then we can lift $e^2_2$ to an oriented 2-cell $\hat{e}^2_2 \subset |\widehat{Y}|$ and $e^3$ to an oriented 3-cell $\hat{e}^3  \subset |\widehat{Y}|$ so that the homology classes $\big( \partial (\hat{e}^2_2)  \cap |\widehat{E}| \big) \otimes 1 \in H_1(C^{\phi}(|\widehat{E}|))$, $\big( \partial (\hat{e}^3) \cap |\widehat{E}| \big) \otimes 1 \in H_2(C^{\phi}(|\widehat{E}|))$  form a basis for $\bigoplus\limits_{i=0}^3 H_i (C^{\phi}(|\widehat{E}|))$. Furthermore, $\tau^{\phi^Y} (Y, \textbf{e}^{Y}, \omega^{|Y|}) = \tau^{\phi} \Big( E, \textbf{e}, \omega; \{ \big( \partial (\hat{e}^2_2)  \cap |\widehat{E}| \big) \otimes 1, \big( \partial (\hat{e}^3) \cap |\widehat{E}| \big) \otimes 1\} \Big)$.
\end{enumerate}
\end{Thm}

\begin{Remark}
Because $H_1^{orb}(Y) \cong H_1^{orb}(E) / \langle \mu^{\alpha} \rangle$, $\phi$ extends to $\phi^Y$ when $\phi(\mu^{\alpha}) =1$.
\end{Remark}
%
%\begin{Remark}
%In Case 2, $\tau^{\phi} \Big( E, \textbf{e}, \omega; \{ \big( \partial (\hat{e}^2_2)  \cap |\widehat{E}| \big) \otimes 1, \big( \partial (\hat{e}^3) \cap |\widehat{E}| \big) \otimes 1\} \Big)$ is $\tau_0$ times the torsion of $C^{\phi}(|\widehat{E}|)$ with respect to the ordered basis $\{ \big( \partial (\hat{e}^2_2)  \cap |\widehat{E}| \big) \otimes 1 \in H_1(C^{\phi}(|\widehat{E}|)), \big( \partial (\hat{e}^3) \cap |\widehat{E}| \big) \otimes 1 \in H_1(C^{\phi}(|\widehat{E}|)) \}$ for $\bigoplus\limits_{i=0}^3 H_i (C^{\phi}(|\widehat{E}|))$, with $\tau_0$ defined as before.
%\end{Remark}

\begin{proof}[Proof of Theorem \ref{orbifold gluing}]
We mimic the argument in the regular case. First endow $|Y|$ with a cell structure that restricts to the preferred cell structure near $\Sigma Y$. Then order the cells in $|Y|$. We assume that the one and two cells in the interior of $(S^1 \times D^2) / \mathbb{Z}_{\alpha}$ satisfy the following: $e^1_1 = \textbf{1} \times int ([0,1])$ is smaller than $e^1_2 = (S^1 - \textbf{1}) \times \textbf{0}$, and  $e^2_1 = (S^1 - \textbf{1}) \times int([0,1])$ is smaller than $e^2_2 = \textbf{1} \times (int(D^2)/ \mathbb{Z}_{\alpha})$. We will need this for later computations. Next orient the cells in $|Y|$ as follows. As before, give each 1-cell in $\Sigma Y$ the orientation of the curve that contains it. In particular, the 1-cell $e^1_2$ inherits the orientation of $S^1 \times \textbf{0}$. Orient $e^2_1$ so that $\partial (e^2_1) = e^1_2$ modulo the 1-cell in $S^1 \times (\partial D^2 / \mathbb{Z}_{\alpha})$. Then orient $e^1_1$ so that $\partial (e^1_1) = e^0 = (1, 0)$ module the 0-cell in $S^1 \times (\partial D^2 / \mathbb{Z}_{\alpha})$. The oriented meridian $\textbf{1} \times (\partial D^2 /  \mathbb{Z}_{\alpha})$ bounds the 2-cell $e^2_2$. We give $e^2_2$ the induced orientation, using the outward last convention for the normal vector. In turn this induces an orientation of the 3-cell $e^3 = (S^1 - \textbf{1}) \times (int(D^2) / \mathbb{Z}_{\alpha})$. Orient the remaining cells in $|Y|$ in an arbitrary way.

Consider the cellular chain complexes $c' = C(|E|, \mathbb{R}), c = C(|Y|, \mathbb{R}), \text{and } c'' = c/c' = C(|Y|, |E|, \mathbb{R})$. Our choices above determine ordered bases for $c', c, \text{and } c''$. Note that these bases are compatible in the sense that for every $i$, the determinant of the matrix that takes the given ordered basis for $c_i$ to the ordered basis gotten by concatenating the ordered basis for $c'_i$ with the ordered basis for $c''_i$ is 1. The homology orientation $\omega$ on $|E|$ induces a homology orientation $\omega^{|Y|}$ on $|Y|$ and a relative homology orientation $\omega^{(|Y|, |E|)}$ on $(|Y|, |E|)$. Choose ordered bases for the homology groups of $c', c, \text{and } c''$ realizing $\omega, \omega^{|Y|}, \text{and } \omega^{(|Y|, |E|)}$, respectively. We can now compute the torsions of $c', c, \text{and } c''$. Let $\tau_0 (c'), \tau_0 (c), \text{and } \tau_0 (c'')$ denote their signs. By \cite[V.1.a, V.2.b]{Tu02}, we get that
\[ \tau_0 (c) = (-1)^{\nu (c,c')+1}\tau_0(c')\tau_0(c''), \]
where 
\begin{equation} \label{nu(c,c')} 
\nu (c, c') = \sum\limits_{i=0}^3 \alpha_i(C'')\alpha_{i-1}(C') \in \mathbb{Z}_2 
\end{equation}
and
\[ 
\alpha_j(C^{\ast})=
\begin{cases} 
      dim(C^{\ast}_0) + \ldots + dim(C^{\ast}_{j}) \in \mathbb{Z}_2 & j \in \{0, 1, 2 , 3 \} \\
      0 \in \mathbb{Z}_2 & j =-1 
   \end{cases}.
\]

\begin{Lem}
$\tau_0 (c'') = -1$.
\end{Lem}

\begin{proof}
$c'' = (\textbf{0} \rightarrow \mathbb{R} \langle e^3 \rangle \xrightarrow{\partial_2} \mathbb{R} \langle e^2_1, e^2_2 \rangle \xrightarrow{\partial_1} \mathbb{R} \langle e^1_1, e^1_2 \rangle \xrightarrow{\partial_0} \mathbb{R} \langle e^0 \rangle \rightarrow \textbf{0})$, with boundary maps given by
\[ \partial_0(e^1_1) = e^0, \partial_0(e^1_2) =0, \]
\[ \partial_1(e^2_1) = e^1_2, \partial_1(e^2_2) =0, \]
\[ \partial_2(e^3) = 0. \]
Note that 
\[ 
H_i (c'')=
\begin{cases} 
      \textbf{0} & i \neq 2, 3\\
       \langle e^2_2 \rangle & i =2\\
       \langle e^3 \rangle & i=3
   \end{cases}
\]
and that $\{e^2_2, e^3\}$ is an ordered basis for $\bigoplus\limits_{i=0}^3 H_i(c'')$ realizing $\omega^{(|Y|, |E|)}$. If $\tau(c'')$ denotes the torsion of $c''$ with respect to $\{e^2_2, e^3\}$, then $\tau(c'') = (-1)^{1} \cdot 1 = -1$. Since $\tau_0(c'')$ is independent of our choice of ordered basis for $\bigoplus\limits_{i=0}^3 H_i(c'')$ realizing $\omega^{(|Y|, |E|)}$, $\tau_0(c'') =-1$.
%(1 \cdot -1 \cdot -1 \cdot 1) =-1$, hence $\tau_0(c'') =-1$.
\end{proof}

\noindent As a result,
\begin{equation} \label{torsionsignSES} 
\tau_0 (c) = (-1)^{\nu (c,c')}\tau_0(c'). 
\end{equation}

Now choose a lift $\{\hat{e}_j\}$ in $|\widehat{E}|$ representing $\textbf{e}$. By fixing a projection of $|\widehat{E}|$ to $|\widehat{Y}|$, we can think of it as a lift in $|\widehat{Y}|$ of the cells in $|E| \subset |Y|$. Lift the cells $e^0, \ldots, e^3$ in the interior of $(S^1 \times D^2) / \mathbb{Z}_{\alpha} \subset |Y|$ to cells $\hat{e}^0, \ldots, \hat{e}^3$ in $|\widehat{Y}|$ so that $\partial (\hat{e}^1_2) = \pm (h-1)\hat{e}^0$, $\partial (\hat{e}^2_1) = \pm (h-1)\hat{e}^1_1 \pm \hat{e}^1_2$ modulo a 1-cell lying over $S^1 \times (\partial D^2 /  \mathbb{Z}_{\alpha})$, and $\partial (\hat{e}^3) = \pm (h-1)\hat{e}^2_2$ modulo a 2-cell lying over $S^1 \times (\partial D^2 /  \mathbb{Z}_{\alpha})$. Assume that $\partial (\hat{e}^2_2) = \pm (\mu-1)\hat{e}^1_1$ modulo a 1-cell lying over $S^1 \times (\partial D^2 /  \mathbb{Z}_{\alpha})$. By definition, $\{\hat{e}_j\} \cup \{\hat{e}^0, \ldots, \hat{e}^3\}$ represents $\textbf{e}^Y$. Each cell in $\{\hat{e}_j\} \cup \{\hat{e}^0, \ldots, \hat{e}^3\}$ inherits an orientation. With it, we have $\partial (\hat{e}^1_2) = (h-1)\hat{e}^0$, $\partial (\hat{e}^2_1) = (1-h)\hat{e}^1_1 + \hat{e}^1_2$ modulo a 1-cell lying over $S^1 \times (\partial D^2 /  \mathbb{Z}_{\alpha})$, $\partial (\hat{e}^3) = (h-1)\hat{e}^2_2$ modulo a 2-cell lying over $S^1 \times (\partial D^2 /  \mathbb{Z}_{\alpha})$, and $\partial (\hat{e}^2_2) = (\mu-1)\hat{e}^1_1$ modulo a 1-cell lying over $S^1 \times (\partial D^2 /  \mathbb{Z}_{\alpha})$. Furthermore, $\{\hat{e}_j\} \cup \{\hat{e}^0, \ldots, \hat{e}^3\}$ inherits an ordering.

Consider the $F$-chain complexes $C' = C^{\phi}(|\widehat{E}|), C = C^{\phi^Y}(|\widehat{Y}|), \text{and } C'' = C/C'$. The orientation and order of the cells in $\{\hat{e}_j\} \cup \{\hat{e}^0, \ldots, \hat{e}^3\}$ determine compatibly ordered bases for $C', C, \text{and }C''$.

\begin{Proofpart1}

\begin{Lem}
$C''$ is acylic and $\tau (C'') = 1$.
\end{Lem}

\begin{proof}
From Observation \ref{Core Observation}, $\hat{e}^0 \otimes 1 = \hat{e}^1_2 \otimes 1 =0$. Then
\[ C'' = (\textbf{0} \rightarrow F \langle \hat{e}^3 \otimes 1 \rangle \xrightarrow{\partial_2 \otimes id} F \langle \hat{e}^2_1 \otimes 1, \hat{e}^2_2 \otimes 1 \rangle \xrightarrow{\partial_1 \otimes id} F \langle \hat{e}^1_1 \otimes 1 \rangle \xrightarrow{\partial_0 \otimes id} \textbf{0}), \]
with boundary maps given by
\[ (\partial_0 \otimes id) (\hat{e}^1_1 \otimes 1) = 0, \]
\[ (\partial_1 \otimes id) (\hat{e}^2_1 \otimes 1) = (1-\phi^Y(h))(\hat{e}^1_1 \otimes 1), (\partial_1 \otimes id) (\hat{e}^2_2 \otimes 1) = (\phi^Y(\mu) -1)(\hat{e}^1_1 \otimes 1), \]
\[ (\partial_2 \otimes id) (\hat{e}^3 \otimes 1) = (\phi^Y(h) -1)(\hat{e}^2_2 \otimes 1) + (\phi^Y(\mu) -1)(\hat{e}^2_1 \otimes 1). \]
Note that $Ker(\partial_1 \otimes id) = \{\alpha (\hat{e}^2_1 \otimes 1) + \alpha (\phi^Y(h) -1)(\phi^Y(\mu) -1)^{-1} (\hat{e}^2_2 \otimes 1) \mid \alpha \in F \}$. Then it is not hard to see that $C''$ is acyclic. By direct computation, $\tau (C'') = 1$.
%
%$\tau (C'') = 1 \cdot (1-\phi^Y(h)) \cdot (1-\phi^Y(h))^{-1} \cdot 1 = 1$.
%Now we check that $\tau (C'') =1$. Note that 
%
%\[ 
%Im(\partial_i \otimes id)=
%\begin{cases} 
     % \textbf{0} & i=0\\
      %\langle im(\hat{e}^2_1 \otimes 1) \rangle & i=1\\
      %\langle im(\hat{e}^3 \otimes 1) \rangle & i =2
   %\end{cases}
%\]
%
\end{proof}

Because $C''$ is acyclic, either $C'$ and $C$ are acyclic or not. If they're not acyclic, then $\tau(C'), \tau(C) =0$, which implies $\tau^{\phi^Y} (Y, \textbf{e}^{Y}, \omega^{|Y|}) = 0 = \tau^{\phi} (E, \textbf{e}, \omega)$, as needed. Suppose $C'$ and $C$ are acyclic. By \cite[V.1.c]{Tu02},
\[ \tau(C) = (-1)^{\nu (C, C')} \tau(C') \tau(C''), \]
where $\nu (C, C')$ is defined as in Equation \ref{nu(c,c')} above. Since $\tau(C'')=1$, this simplifies to
\begin{equation} \label{torsionacyclicSES}
\tau(C) = (-1)^{\nu (C, C')} \tau(C').
\end{equation}
Multiplying Equation \ref{torsionacyclicSES} by Equation \ref{torsionsignSES} gives
\begin{equation} \label{torsionacyclicSES2} 
\tau_0 (c) \tau(C) = (-1)^{\nu (C, C')} (-1)^{\nu (c,c')}  \tau_0(c') \tau(C'). 
\end{equation}
It's easy to check that $\nu (C, C') = \nu (c,c') \in \mathbb{Z}_2$. Then Equation \ref{torsionacyclicSES2} becomes
\[ \tau_0 (c) \tau(C) = \tau_0(c') \tau(C'). \]
By definition, $\tau_0 (c) \tau(C) = \tau^{\phi^Y} (Y, \textbf{e}^{Y}, \omega^{|Y|})$ and $\tau_0(c') \tau(C') = \tau^{\phi} (E, \textbf{e}, \omega)$, so this concludes the proof of Case 1.

\end{Proofpart1}

\begin{Proofpart2}

\begin{Lem}
$C''$ is acylic and $\tau (C'') = (\phi^Y (h) -1)^{-1}$.
\end{Lem}

\begin{proof}
$C'' = (\textbf{0} \rightarrow F \langle \hat{e}^3 \otimes 1 \rangle \xrightarrow{\partial_2 \otimes id} F \langle \hat{e}^2_1 \otimes 1, \hat{e}^2_2 \otimes 1 \rangle \xrightarrow{\partial_1 \otimes id} F \langle \hat{e}^1_1 \otimes 1, \hat{e}^1_2 \otimes 1 \rangle \xrightarrow{\partial_0 \otimes id} F \langle \hat{e}^0 \otimes 1 \rangle \rightarrow \textbf{0}),$ with boundary maps given by:
\[ (\partial_0 \otimes id) (\hat{e}^1_1 \otimes 1) =  \hat{e}^0 \otimes 1, (\partial_0 \otimes id) (\hat{e}^1_2 \otimes 1) =  (\phi^Y(h) -1) (\hat{e}^0 \otimes 1),\]
\[ (\partial_1 \otimes id) (\hat{e}^2_1 \otimes 1) = (1-\phi^Y(h))(\hat{e}^1_1 \otimes 1) + \hat{e}^1_2 \otimes 1, (\partial_1 \otimes id) (\hat{e}^2_2 \otimes 1) =0, \]
\[ (\partial_2 \otimes id) (\hat{e}^3 \otimes 1) = (\phi^Y(h) -1)(\hat{e}^2_2 \otimes 1). \]
Note that $ Ker(\partial_0 \otimes id) = \{ \alpha (\hat{e}^1_1 \otimes 1) + \alpha (1- \phi^Y(h))^{-1} (\hat{e}^1_2 \otimes 1) \mid \alpha \in F \} $. Then it is not hard to verify that $C''$ is acyclic. By direct computation, $\tau (C'') = (\phi^Y(h) -1)^{-1} $.
%
%$\tau (C'') = 1 \cdot -1 \cdot -(\phi^Y(h) -1)^{-1} \cdot 1 = (\phi^Y(h) -1)^{-1} $.
%
%We now compute $\tau (C'')$. Note that 
%
%\[
%Im (\partial_i \otimes id) = 
%\begin{cases} 
    %  \langle im(\hat{e}^1_1 \otimes 1) \rangle & i=0\\
     %  \langle im(\hat{e}^2_1 \otimes 1) \rangle & i =1\\
     %  \langle im(\hat{e}^3 \otimes 1) \rangle & i=2
  % \end{cases}
%\]
%
\end{proof}

As in Case 1, we can assume $C'$ and $C$ are acyclic. Again by \cite[V.1.c]{Tu02},
\[ \tau(C) = (-1)^{\nu (C, C')} \tau(C') \tau(C''), \]
where $\nu (C, C')$ is defined as above.  Since $\tau(C'') = (\phi^Y(h) -1)^{-1}$, this becomes 
\begin{equation}\label{case2version}
\tau(C) = (-1)^{\nu (C, C')} \tau(C') (\phi^Y(h) -1)^{-1}.
\end{equation}
Multiplying Equation \ref{case2version} by Equation \ref{torsionsignSES} gives
\begin{equation}\label{case2versionwithsign}
\tau_0 (c) \tau(C) = (-1)^{\nu (C, C')} (-1)^{\nu (c,c')}  \tau_0(c') \tau(C') (\phi^Y(h) -1)^{-1}.
\end{equation}
Since $\nu (C, C') = \nu (c,c')$, Equation \ref{case2versionwithsign} becomes
\[ \tau_0 (c) \tau(C) = \tau_0(c') \tau(C') (\phi^Y(h) -1)^{-1}. \]
This implies $\tau^{\phi^Y} (Y, \textbf{e}^{Y}, \omega^{|Y|}) = \tau^{\phi} (E, \textbf{e}, \omega) (\phi^Y(h) -1)^{-1}$, as needed.

\end{Proofpart2} 

\begin{Proofpart3}

\[ C'' = (\textbf{0} \rightarrow F \langle \hat{e}^3 \otimes 1 \rangle \xrightarrow{\partial_2 \otimes id} F \langle \hat{e}^2_1 \otimes 1, \hat{e}^2_2 \otimes 1 \rangle \xrightarrow{\partial_1 \otimes id} F \langle \hat{e}^1_1 \otimes 1, \hat{e}^1_2 \otimes 1 \rangle \xrightarrow{\partial_0 \otimes id} F \langle \hat{e}^0 \otimes 1 \rangle \rightarrow \textbf{0}),\]
with boundary maps given by:
\[ (\partial_0 \otimes id) (\hat{e}^1_1 \otimes 1) =  \hat{e}^0 \otimes 1, (\partial_0 \otimes id) (\hat{e}^1_2 \otimes 1) = 0,\]
\[ (\partial_1 \otimes id) (\hat{e}^2_1 \otimes 1) = \hat{e}^1_2 \otimes 1, (\partial_1 \otimes id) (\hat{e}^2_2 \otimes 1) =0, \]
\[ (\partial_2 \otimes id) (\hat{e}^3 \otimes 1) = 0. \]
Note that
\[
H_i (C'')=
\begin{cases} 
      \textbf{0} & i \neq 2, 3\\
       \langle \hat{e}^2_2 \otimes 1 \rangle & i =2\\
       \langle \hat{e}^3 \otimes 1 \rangle & i=3.
   \end{cases}
\]
We fix the ordered basis in $\bigoplus\limits_{i=0}^3 H_i (C'')$ to be $\{ \hat{e}^2_2 \otimes 1, \hat{e}^3 \otimes 1 \}$. Let $\tau(C'')$ denote the resulting torsion of $C''$. By direct computation, $\tau(C'') = (-1)^1 \cdot 1 = -1$.

Using the long exact sequence $\mathcal{H}$ for the pair $(C, C')$, our computation of $H_i (C'')$, and the assumption that $C$ is acyclic, we get that
\[
H_i (C')=
\begin{cases} 
      \textbf{0} & i \neq 1, 2\\
       \langle \delta_i(\hat{e}^2_2 \otimes 1) \rangle = \langle \big( \partial (\hat{e}^2_2)  \cap |\widehat{E}| \big) \otimes 1 \rangle & i =1\\
       \langle \delta_i(\hat{e}^3 \otimes 1) \rangle = \langle \big( \partial (\hat{e}^3) \cap |\widehat{E}| \big) \otimes 1 \rangle & i =2.
   \end{cases}
\]
where $\delta_i$ is the connecting homomorphism $H_{i+1}(C'') \rightarrow H_{i}(C')$. We fix the ordered basis in $\bigoplus\limits_{i=0}^3 H_i (C')$ to be $\{\big( \partial (\hat{e}^2_2)  \cap |\widehat{E}| \big) \otimes 1, \big( \partial (\hat{e}^3) \cap |\widehat{E}| \big) \otimes 1\}$ and denote the resulting torsion of $C'$ by $\tau(C')$.

With the above bases, $\mathcal{H}$ becomes a based acyclic chain complex. Set
\[
\tau(C' \subset C) = (-1)^{\theta (C, C')} \tau (\mathcal{H}) \in F,
\]
where
\[
\theta (C, C') = \sum\limits_{i=0}^3 \Big(\big(\beta_i(C) +1\big)\big(\beta_i(C') + \beta_i(C'')\big) + \beta_{i-1}(C')\beta_i(C'') \Big) \in \mathbb{Z}_2
\]
and
\[ 
\beta_j(C^{\ast}) =
\begin{cases} 
      dim(H_0(C^{\ast})) + \ldots + dim(H_j(C^{\ast})) \in \mathbb{Z}_2 & j \in \{0, 1, 2 , 3 \} \\
      0 \in \mathbb{Z}_2 & j =-1 
   \end{cases}.
\]
It is not hard to verify that $\tau(C' \subset C) = (-1)^{1} \cdot 1 =-1$.

By \cite[V.1.a]{Tu02},
\begin{equation}\label{Case3gluing}
\tau(C) = (-1)^{\nu(C, C')} \tau(C') \tau (C'') \tau(C' \subset C) = (-1)^{\nu(C, C')} \tau(C').
\end{equation}
Multiplying Equation \ref{Case3gluing} by Equation \ref{torsionsignSES} gives
\begin{equation} \label{Case3gluingsimplified}
\tau_0(c) \tau(C) = (-1)^{\nu(C, C')} (-1)^{\nu(c, c')} \tau_0(c') \tau(C').
\end{equation}
Since $\nu(C, C')=\nu(c, c')$, Equation \ref{Case3gluingsimplified} becomes 
\[
\tau_0(c) \tau(C) = \tau_0(c') \tau(C').
\]
By definition, 
\[
\tau_0(c) \tau(C) = \tau^{\phi^Y} (Y, \textbf{e}^{Y}, \omega^{|Y|})
\]
and 
\[\tau_0(c') \tau(C') = \tau^{\phi} \Big( E, \textbf{e}, \omega; \{ \big( \partial (\hat{e}^2_2)  \cap |\widehat{E}| \big) \otimes 1, \big( \partial (\hat{e}^3) \cap |\widehat{E}| \big) \otimes 1\} \Big),
\]
so this concludes the proof of Case 3. 
%
%\[
%\mathcal{H}= (\textbf{0} \rightarrow H_3(C') \rightarrow \textbf{0} \rightarrow F \langle \hat{e}^3 \otimes 1 \rangle \xrightarrow{\delta} H_2(C') \rightarrow \textbf{0} \rightarrow F \langle \hat{e}^2_2 \otimes 1 \rangle \xrightarrow{\delta} H_1(C') \rightarrow \textbf{0}).
%\]
%
\end{Proofpart3}
\end{proof}

The following gluing formulas generalize Theorem \ref{gluing lemmas}.

\begin{Thm}\label{gen}
Let $E$ be a compact, connected, oriented 3-orbifold with $\Sigma E$ an oriented link and $\partial E$ a union of tori. Glue a solid torus $S^1 \times D^2$ to $E$ along a component of $\partial E$. We get a 3-orbifold $Y$ with $\Sigma Y = \Sigma E$. Fix an orbifold Euler structure $\textbf{e}$ on $E$ and a homology orientation $\omega$ on $|E|$. This induces an orbifold Euler structure $\textbf{e}^Y$ on $Y$ and a homology orientation $\omega^{|Y|}$ on $|Y|$. Orient $S^1 \times \textbf{0}$, and let $h \in H_1^{orb}(Y)$ denote the induced homology class. Let $F$ be a field, and let $\phi: \mathbb{Z}[H_1^{orb}(E)] \rightarrow F$ be a ring homomorphism that extends to a ring homomorphism $\phi^Y : \mathbb{Z}[H_1^{orb}(Y)] \rightarrow F$. We have a couple of cases:

\begin{enumerate}
\item Suppose $\phi^Y (h ) \neq 1$. Then $\tau^{\phi^Y} (Y, \textbf{e}^{Y}, \omega^{|Y|}) = \tau^{\phi} (E, \textbf{e}, \omega) \cdot (\phi^Y (h) -1)^{-1}$.
\item Suppose $\phi^Y (h ) = 1$. Suppose further that $C^{\phi^Y}(|\widehat{Y}|)$ is acyclic. Let $e^2_2$ denote the 2-cell in $S^1 \times D^2$. Orient $e^2_2$ so that $\partial (e^2_2) = \textbf{1} \times \partial D^2$. Let $e^3$ denote the 3-cell $(S^1 - \textbf{1}) \times int(D^2)$ in $S^1 \times D^2$. Give $e^3$ the product orientation. Then we can lift $e^2_2$ to an oriented 2-cell $\hat{e}^2_2 \subset |\widehat{Y}|$ and $e^3$ to an oriented 3-cell $\hat{e}^3  \subset |\widehat{Y}|$ so that the homology classes $\big( \partial (\hat{e}^2_2)  \cap |\widehat{E}| \big) \otimes 1 \in H_1(C^{\phi}(|\widehat{E}|))$, $\big( \partial (\hat{e}^3) \cap |\widehat{E}| \big) \otimes 1 \in H_2(C^{\phi}(|\widehat{E}|))$  form a basis for $\bigoplus\limits_{i=0}^3 H_i (C^{\phi}(|\widehat{E}|))$. Furthermore, $\tau^{\phi^Y} (Y, \textbf{e}^{Y}, \omega^{|Y|}) = \tau^{\phi} \Big( E, \textbf{e}, \omega; \{ \big( \partial (\hat{e}^2_2)  \cap |\widehat{E}| \big) \otimes 1, \big( \partial (\hat{e}^3) \cap |\widehat{E}| \big) \otimes 1\} \Big)$.
\end{enumerate}
\end{Thm}

\begin{Remark}
Orient $\textbf{1} \times \partial D^2 \subset S^1 \times D^2$. Let $\mu \in H_1^{orb}(E)$ denote its induced homology class. Because $H_1^{orb}(Y) \cong H_1^{orb}(E) / \langle \mu \rangle$, $\phi$ extends to $\phi^Y$ when $\phi(\mu) =1$.
\end{Remark}

\begin{proof}[Proof of Theorem \ref{gen}]
Analogous to the proof of Theorem \ref{orbifold gluing}.
\end{proof}

\section{Consequences}\label{Sec: Apps}

First we use the gluing formulas to determine how (some of) the components of the orbifold Turaev torsion invariant change when we remove a curve from the singular set.

\begin{Thm}\label{app1}
Let $Y$ be a compact, connected, oriented 3-orbifold with $\Sigma Y$ an oriented link $L_1 \cup \ldots \cup L_k$. Let $Y'$ be the 3-orbifold gotten by removing $L_k$ from $\Sigma Y$. Let $E$ denote the exterior of $L_k$ in $|Y|=|Y'|$. Note that $E$ inherits the structure of a 3-orbifold with $\Sigma E = L_1 \cup \ldots \cup L_{k-1}$. Fix an orbifold Euler structure $\textbf{e}$ on $E$ and a homology orientation $\omega$ on $|E|$. This induces orbifold Euler structures $\textbf{e}^Y , \textbf{e}^{Y'}$ on $Y, Y'$, respectively, and a homology orientation $\omega^{|Y|} = \omega^{|Y'|}$ on $|Y|=|Y'|$. Let $F$ be a field, and let $\phi^Y : \mathbb{Z}[H_1^{orb}(Y)] \rightarrow F$ be a ring homomorphism that extends to a ring homomorphism $\phi^{Y'} : \mathbb{Z}[H_1^{orb}(Y')] \rightarrow F$. Then $\tau^{\phi^Y}(Y, \textbf{e}^{Y}, \omega^{|Y|}) = \tau^{\phi^{Y'}}(Y', \textbf{e}^{Y'}, \omega^{|Y'|})$.
\end{Thm}

\begin{Remark}
Let $\alpha_k$ denote the multiplicity of $L_k$ in $Y$. Orient the meridian of $L_k$ and let $\mu_k$ denote its homology class in $H_1^{orb}(E)$. Because $H_1^{orb}(Y) \cong H_1^{orb}(E) / \langle \mu_k^{\alpha_k} \rangle$ and $H_1^{orb}(Y') \cong H_1^{orb}(E) / \langle \mu_k \rangle$, $\phi^{Y}$ extends to $\phi^{Y'}$ when $\phi^{Y}(\mu_k) =1$.
\end{Remark}

\begin{proof}[Proof of Theorem \ref{app1}]
Note that $\phi^Y(L_k) = \phi^{Y'}(L_k)$ and that the chain complex $C^{\phi ^Y}(|\widehat{Y}|)$ is acyclic if and only if the chain complex $C^{\phi ^{Y'}}(|\widehat{Y'}|)$ is acyclic. Then use Theorem \ref{orbifold gluing} and Theorem \ref{gen}.
\end{proof}
%
%\begin{Thm}\label{CorGluing}
%Let $Y$ be a closed, connected, orientable 3-orbifold with $\Sigma Y$ an oriented knot $K$. Let $E$ denote the exterior of $K$. Note that $E$ inherits the structure of a compact, connected, orientable 3-manifold with torus boundary. Let $F$ be a field, and let $\phi^Y : \mathbb{Z}[H_1^{orb}(Y)] \rightarrow F$ be a ring homomorphism. Let $\mu$ denote the meridian of $K$. Orient $\mu$ so that $lk(K, \mu)=1$. Suppose $\phi^Y (\mu ) =1$. We have the composition $\phi^Y \circ q : \mathbb{Z}[H_1(E)] \rightarrow F$, where $q$ is induced by the quotient map $H_1(E) \rightarrow H_1^{orb}(Y)$. $\phi^Y \circ q$ induces a map $\phi^{|Y|} : \mathbb{Z}[H_1(|Y|)] \rightarrow F$. Then for every Euler structure $\textbf{e}$ and homology orientation $\omega$ on $E$, $\tau^{\phi^Y}(Y, \textbf{e}^{Y}, \omega^{|Y|}) = \tau^{\phi^{|Y|}}(|Y|, \textbf{e}^{|Y|}, \omega^{|Y|})$.
%\end{Thm}
%
%\begin{proof}
%Note that $\phi^Y(K)= \phi^{|Y|}(K)$. By the gluing formulas in Theorem \ref{gluing lemmas} and Theorem \ref{orbifold gluing}, we're done.
%\end{proof}

Next we give a formula relating the Turaev torsion invariant of the orbifold to the Turaev torsion invariant of the underlying space, in the case when the singular set is a nullhomologous knot. 

\begin{Thm}\label{thm: appplication}
Let $Y$ be a closed, connected, oriented 3-orbifold with $\Sigma Y$ an oriented and nullhomologous knot $K$. Suppose $b_1(|Y|) \geq 1$. Let $\alpha$ denote $K$'s multiplicity. Let $E$ denote the exterior of $K$. Then there is a surjective, $\alpha$ to 1 map $f: Eul(Y) \rightarrow Eul(|Y|)$ and a ring homomorphism $g: \mathbb{Z}[H_1(E)] \rightarrow Q(\mathbb{Z}[H_1^{orb}(Y)])$ so that for every Euler structure $\textbf{e}$ and homology orientation $\omega$ on $E$, $\tau(Y, \textbf{e}^{Y}, \omega^{|Y|}) = \tau(|Y|, f(\textbf{e}^Y), \omega^{|Y|}) + g(\tau(E, \textbf{e}, \omega)) \in Q(\mathbb{Z}[H_1^{orb}(Y)])$. 
\end{Thm}
%
%\begin{Remark}
%This is not true in general. For example, take $Y = (S^2, 2, 3, 5) \times S^1$. Then $\Sigma Y$ consists of three curves, none of which is nullhomologous, and $H_1^{orb}(Y) \cong \mathbb{Z} \cong H_1(|Y|)$.
%end{Remark}

The proof employs the following straightforward lemma:

\begin{Lem} \label{FinalLem}
Let $\mu$ denote the meridian of $K$. Orient $\mu$ so that $lk(K, \mu)=1$. Then $H_1^{orb}(Y) \cong H_1(|Y|) \oplus (\langle \mu \rangle / \langle \mu^{\alpha} \rangle)$.
\end{Lem}

\begin{proof}
Since $K$ is nullhomologous in $|Y|$, $\mu$ has infinite order in $H_1(E)$ by the half-lives, half-dies principle. Then we get the following short exact sequence:
\[\textbf{1} \rightarrow \langle e^2_2 \rangle \xrightarrow{\delta} H_1(E) \rightarrow H_1(|Y|) \rightarrow \textbf{1}, \]
where $e^2_2$ is the oriented meridional disk with $\partial (e^2_2) = \mu$. Note that $\delta(e^2_2) = \mu$. Now pick a compact, connected, oriented surface in $|Y|$ bounded by $K$. This gives a left splitting $H_1(E) \rightarrow \langle e^2_2 \rangle $. As a result, the short exact sequence splits, and we get that $H_1(E) \cong H_1(|Y|) \oplus \langle \mu \rangle$. This implies that $H_1^{orb}(Y) \cong H_1(E)/  \langle \mu^{\alpha} \rangle \cong H_1(|Y|) \oplus (\langle \mu \rangle / \langle \mu^{\alpha} \rangle)$.
\end{proof}

\begin{proof}[Proof of Theorem \ref{thm: appplication}]

We have a canonical splitting $\psi^Y: Q \Big( \mathbb{Z} [H_1^{orb}(Y)] \Big) \rightarrow \bigoplus\limits_{l=1}^r F_l$. Recall that each $F_l$ is the quotient field 
\[
Q \Big( \mathbb{Q}(\zeta_{n_l}) \big[ H_1^{orb}(Y) /Tor \big(H_1^{orb}(Y) \big) \big] \Big)
\]
of the group algebra 
\[\mathbb{Q}(\zeta_{n_l}) \big[ H_1^{orb}(Y) /Tor \big(H_1^{orb}(Y) \big) \big]
\]
over a cyclotomic field $\mathbb{Q}(\zeta_{n_l})$, and the cyclotomic fields are gotten by looking at isomorphism classes of characters of $Tor (H_1^{orb}(Y) )$. By Lemma \ref{FinalLem}, we have that 
\[
H_1^{orb}(Y) /Tor \big(H_1^{orb}(Y) \big) \cong H_1(|Y|) /Tor \big(H_1(|Y|) \big).
\]
Hence we can think of each $F_l$ as
\[
Q \Big( \mathbb{Q}(\zeta_{n_l}) \big[ H_1(|Y|) /Tor \big(H_1(|Y|) \big) \big] \Big).
\]
For each $l$, let $\psi^Y_l$ denote the composition
\[
\mathbb{Z} [H_1^{orb}(Y)]  \xrightarrow{I} Q \Big( \mathbb{Z} [H_1^{orb}(Y)] \Big) \xrightarrow{\psi^Y} \bigoplus\limits_{l=1}^r F_l \xrightarrow{\pi_l} F_l.
\]
Without loss of generality, assume that $\psi^Y_l(\mu) =1$ for $l \in \{1, \ldots, m\}$, and otherwise for $l \in \{m+1, \ldots, r\}$. Because $Tor (H^{orb}_1(Y)) \cong Tor (H_1(|Y|)) \oplus \langle \mu \mid \mu^{\alpha} =1 \rangle$, we can think of $\bigoplus\limits_{l=1}^m F_l$ as the canonical splitting of $Q \Big( \mathbb{Z} [H_1(|Y|)] \Big)$. Then by Theorem \ref{app1}, we have that for every $l \in \{1, \ldots, m\}$,
\[
\tau^{\psi^Y_l}(Y, \textbf{e}^{Y}, \omega^{|Y|}) = \tau^{\psi^{|Y|}_l}(|Y|, \textbf{e}^{|Y|}, \omega^{|Y|}),
\]
where $\psi^{|Y|}_l$ is the ring homomorphism $\mathbb{Z} [H_1(|Y|)] \rightarrow F_l$ induced by the composition
\[
\mathbb{Z}[H_1(E)] \xrightarrow{q} \mathbb{Z} [H_1^{orb}(Y)] \xrightarrow{\psi^Y_l} F_l,
\]
and $q$ is induced by the quotient map $H_1(E) \rightarrow H_1^{orb}(Y)$. Now let $l \in \{ m+1, \ldots, r\}$. Note that $\psi^Y_l(\mu) \neq 1$. By Theorem \ref{orbifold gluing} and Theorem \ref{ThmTu97},
\[
\tau^{\psi^Y_l}(Y, \textbf{e}^{Y}, \omega^{|Y|}) = \tau^{(\psi^{Y}_l \circ q)}(E, \textbf{e}, \omega) = (\psi^{Y}_l \circ q) (\tau (E, \textbf{e}, \omega)),
\]
since $b_1(|Y|) \geq 1 \Rightarrow b_1(E) \geq 2$. Hence if we let $ g = (\psi^{Y}_{m+1} + \ldots + \psi^{Y}_{r}) \circ q$, we get that
\[
\tau(Y, \textbf{e}^{Y}, \omega^{|Y|}) = \tau(|Y|, \textbf{e}^{|Y|}, \omega^{|Y|}) + g(\tau(E, \textbf{e}, \omega)) \in Q(\mathbb{Z}[H_1^{orb}(Y)]).
\]

We finish by defining $f: Eul(Y) \rightarrow Eul(|Y|)$. Let $\overline{E} ^{orb}$ denote the cover of $E$ with deck group $H_1(E) / \langle \mu^{\alpha} \rangle$. Let $\overline{E}$ denote the cover of $E$ with deck group $H_1(E) / \langle \mu \rangle$. Think of $\overline{E} ^{orb}$ as $\widehat{E} / \langle \mu^{\alpha} \rangle$ and $\overline{E}$ as $\widehat{E} / \langle \mu \rangle$. We get a projection map $\overline{f}: \overline{E} ^{orb} \rightarrow \overline{E}$ that is equivariant with respect to the canonical map $i: H_1(E) / \langle \mu^{\alpha} \rangle \rightarrow H_1(E) / \langle \mu \rangle$, and commutes with the projection maps $q_1: \widehat{E} \rightarrow \overline{E} ^{orb}$ and $q_2: \widehat{E} \rightarrow \overline{E}$. Specifically, $\overline{f} (h \cdot x) = i(h) \cdot \overline{f}(x)$ and $q_2 = \overline{f} \circ q_1$. Canonically extend $\overline{f}$ to a projection map $\hat{f} : |\widehat{Y}| \rightarrow \widehat{|Y|}$ that is equivariant with respect to $i$. Define $f: Eul(Y) \rightarrow Eul(|Y|)$ to be the induced function. Since $q_2 = \overline{f} \circ q_1$, we have that $f(\textbf{e}^Y) = \textbf{e}^{|Y|}$. It's not hard to see that $f$ is surjective. Finally, $f$ is $\alpha$ to 1 because we have this commutative diagram
\begin{center}
$\begin{CD}
H_1(E) / \langle \mu^{\alpha} \rangle @>i>> H_1(E) / \langle \mu \rangle \\
@VV\cong V @AA\cong A,\\
H_1(|Y|) \times \langle \mu \mid \mu^{\alpha} \rangle @>\pi_1>> H_1(|Y|)
\end{CD}$
\end{center}
$f$ is equivariant with respect to $i$, and $H_1(E) / \langle \mu^{\alpha} \rangle, H_1(E) / \langle \mu \rangle$ act freely and transitively on $Eul(Y), Eul(|Y|)$, respectively.
\end{proof}

\begin{Remark}
The map $g$ is not injective, but it's not hard to see that in some cases $\tau(E, \textbf{e}, \omega)$ can be recovered from $g(\tau(E, \textbf{e}, \omega))$. This shows that the orbifold Turaev torsion invariant can be used to detect orbifold structures in contrast to the orbifold Seiberg-Witten invariant.
\end{Remark}

\end{document}